\documentclass[10pt,a4paper,reqno]{amsart}
\usepackage{amsfonts}
\usepackage{amssymb}
\usepackage{amsmath}
\usepackage{amsthm}
\usepackage{cite}
\usepackage{enumitem}
\usepackage{tikz}
\usepackage{subfigure}
\usepackage{stackrel}
\usepackage{hyperref} 
\usepackage{dsfont}
\theoremstyle{plain}
\newtheorem{thm}{Theorem} 
\newtheorem{defn}[thm]{Definition}
\newtheorem{lem}[thm]{Lemma}
\newtheorem{prop}[thm]{Proposition}

\newtheorem{rem}[thm]{Remark}

\providecommand{\ind}{\mathds{1}} 
\providecommand{\sm}{\setminus}
\providecommand{\N}{\mathbb{N}}
\providecommand{\R}{\mathbb{R}} 
\providecommand{\Z}{\mathbb{Z}}

\providecommand{\cK}{\mathcal{K}}

\providecommand{\cH}{\mathcal{H}}
\providecommand{\eps}{\varepsilon}
\providecommand{\ov}{\overline}

\providecommand{\wto}{\rightharpoonup}
\providecommand{\skp}[2]{\langle#1,#2\rangle}
\providecommand{\les}{\lesssim}

\DeclareMathOperator{\supp}{supp}

\DeclareMathOperator{\loc}{loc}
\DeclareMathOperator{\curl}{curl}
\DeclareMathOperator{\Eig}{Eig}

\DeclareMathOperator{\diver}{div}
\DeclareMathOperator{\codim}{codim}
\DeclareMathOperator{\dist}{dist}
\DeclareMathOperator{\spa}{span}

\renewcommand{\qed}{\hfill $\Box$}

\newcommand{\vecIII}[3]{
\ensuremath{
\begin{pmatrix}
#1 \\ #2 \\ #3 \\
\end{pmatrix}}}

\renewcommand{\r}[1]{\textcolor{red}{#1}}

\usepackage{color}
\definecolor{Darkgblue}{rgb}{0.3,0.3,0.5}

\begin{document}

\allowdisplaybreaks

\title{Dual variational methods for Time-harmonic Nonlinear Maxwell's Equations}

\author{Rainer Mandel}
%\author{Rainer Mandel\textsuperscript{1}}
%\address{\textsuperscript{1}Karlsruhe
%Institute of Technology, Institute for Analysis, Englerstra{\ss}e 2, 76131 Karlsruhe, Germany}
%\email{lucrezia.cossetti@kit.edu; rainer.mandel@kit.edu}
  
\subjclass[2020]{35J60, 35Q61}
%35B45: A priori estimates in the context of PDEs
%35J05: Laplace operator, Helmholtz equation (reduced wave equation), Poisson equation
%35Q61: Maxwell equations

\keywords{}
\date{\today}   

\begin{abstract}
  We prove the existence of infinitely many nontrivial solutions for time-harmonic nonlinear Maxwell's
  equations on bounded domains and on $\R^3$ using dual variational methods. In the dual setting we apply a new version
  of the Symmetric Mountain Pass Theorem that does not require the Palais-Smale condition.
\end{abstract}

\maketitle
\allowdisplaybreaks
\setlength{\parindent}{0cm}

\section{Introduction}

Nonlinear boundary value problems  of the form
\begin{align}\label{eq:NLCurlCurlD}
  \begin{aligned}
  \nabla\times \big(\mu(x)^{-1}\nabla\times E\big)  - \omega^2 \eps(x) E =  f(x,E) \quad\text{in
  }\Omega,\qquad E\times \nu = 0  \quad\text{on }\partial\Omega.
  \end{aligned}
\end{align}
originate from Maxwell's equations for time-harmonic electric field $E(x)e^{i\omega t}$  propagating in an
optically nonlinear medium. Here,  $\omega\in\R$ is the frequency of the wave, $\mu$ denotes the
permeability matrix and $\eps$ is the permittivity matrix of the propagation medium and $\nu$ denotes the
outer unit normal field of the domain $\Omega\subset\R^3$.
The nonlinearity $f(x,E)\in\R^3$ represents the superlinear part of the electric displacement field within the
propagation medium, see~\cite[pp. 825-826]{Med_GS}.   Several existence results for nontrivial solutions of
nonlinear Maxwell boundary value problems like~\eqref{eq:NLCurlCurlD} have been
proved \cite{BarMed,BarMed_NLMaxDomains,YangYeZhang,BarMedIII} under various assumptions on the data, notably
for bounded $C^2-$domains $\Omega\subset\R^3$ and the model nonlinearity $f(x,E)=|E|^{p-2}E$ for $2<p<6$ with
uniformly positive definite matrices $\eps,\mu$ such that $\eps\in W^{1,\infty}(\Omega;\R^{3\times 3})$, see
Theorem~2.2 and Proposition~3.1 in~\cite{BarMed}. In this paper we set up an alternative approach  by
implementing the dual variational method for
\begin{align}\label{eq:NLCurlCurlN}
  \begin{aligned}
  \nabla\times \big(\mu(x)^{-1}\nabla\times E\big) - \omega^2\eps(x) E =  f(x,E) \quad\text{in }\Omega,\qquad 
   \mu(x)^{-1} (\nabla\times E)\times \nu = 0 \quad\text{on }\partial\Omega.
\end{aligned} 
\end{align}  
By analogy with classical elliptic boundary value problems we will call 
\eqref{eq:NLCurlCurlD} a Dirichlet problem and \eqref{eq:NLCurlCurlN} a Neumann problem. We
refer to Remark~\ref{rem:BC} for a justification of this nomenclature. Given that the Dirichlet problem has
already been studied to some extent, we focus on the Neumann problem in the following.

\medskip

In our first main result we show that \eqref{eq:NLCurlCurlN} has a ground state and infinitely many bound
state solutions under the following assumptions on the data:
  \begin{itemize}
    \item[(A1)] $\Omega$ is a bounded $C^1$-domain satisfying an exterior ball condition.
    \item[(A2)]  $\eps,\mu\in L^\infty(\Omega;\R^{3\times 3})$ are uniformly
    positive definite with $\eps\in W^{1,3}(\Omega;\R^{3\times 3})$.     
    \item[(A3)] $f:\Omega\times\R^3\to\R^3$ is measurable with $f(x,E)=f_0(x,|E|)|E|^{-1}E$ 
    where, for almost all $x\in\Omega$,
    \begin{align*}
       s&\mapsto f_0(x,s)  \text{ is positive, differentiable and increasing on } (0,\infty),\\
       s&\mapsto s^{-1} f_0(x,s) \text{ is increasing on }(0,\infty)
    \end{align*}
    and there are $c_1,c_2>0$  and $2<p<6$ such that
    \begin{align*}
    \frac{1}{2}f_0(x,s)s - \int_0^s f_0(x,t)\,dt
    \geq c_1 s^p  
    \geq c_2 f_0(x,s) s 
    \quad\text{for all }s\geq 0.
  \end{align*} 
  \end{itemize}
  %Of course, $f(x,E)=|E|^{p-2}E$ is the model nonlinearity for (A3). 
  As we explain further below,
  the $W^{1,3}$-regularity for $\eps$ is needed to ensure Sobolev-type embeddings of the function spaces we
  are working in. To find solutions
  of~\eqref{eq:NLCurlCurlN} under the given assumptions, it is reasonable to perform a  Helmholtz
  decomposition where a given vector field $E$ is splitted according to $E=E_1+E_2$ where $\eps E_1$ is
  divergence-free and $E_2$ is curl-free in a suitable sense. To make this rigourous we introduce
  the Hilbert space $\mathcal H:=H(\curl;\Omega)$ as the completion of $C^\infty(\Omega;\R^3)$ with respect to
  the inner product
  \begin{equation}\label{eq:innerproduct}
    \skp{E}{F} := \int_\Omega\mu(x)^{-1}(\nabla\times E)\cdot (\nabla\times F) + \eps(x)E\cdot F\,dx.
  \end{equation}
  The appropriate function space for \eqref{eq:NLCurlCurlN}
  turns out to be $\mathcal V\oplus \mathcal W$ where
  \begin{align*}
      \mathcal V := \Big\{ E_1\in \mathcal H: \int_\Omega
      \eps(x)E_1\cdot\nabla \Phi\,dx = 0 \text{ for all } \Phi\in C^1(\ov\Omega)\Big\},  \qquad
      \mathcal W:= \big\{ \nabla u : u\in W^{1,p}(\Omega)\big\}. 
  \end{align*}
  Note that $\mathcal V,\mathcal W$ are formally orthogonal to each other and that the curl operator vanishes
  identically on $\mathcal W$.
  The Sobolev-type embeddings of $\mathcal V$ that we shall prove later   
  ensure  that the associated Euler functional 
  \begin{equation}\label{eq:defI}
    I(E) := \frac{1}{2} \int_{\Omega} \mu(x)^{-1}(\nabla\times E_1)\cdot (\nabla\times E_1)\,dx 
    -\frac{\omega^2}{2}\int_\Omega \eps(x)E\cdot E\,dx
%    \b{-\frac{\omega^2}{2}\int_\Omega (E_1^T\eps(x)E_1+E_3^T\eps(x)E_2)\,dx}
    -  \int_{\Omega} F(x,E)\,dx
  \end{equation} 
  is continuously differentiable where $E=E_1+E_2$ with $E_1\in\mathcal V,E_2\in \mathcal W$. Here,
  $F(x,\cdot)$ denotes the primitive of $f(x,\cdot)$ with $F(x,0)=0$. 
  A weak solution  $E\in\mathcal V\oplus \mathcal W$ of \eqref{eq:NLCurlCurlN} is then defined as a solution
  of the Euler-Lagrange equation $I'(E)=0$. A nontrivial weak solution having least energy among all
  nontrivial weak solutions is called a ground state. Our first result reads as follows.
  
  \begin{thm} \label{thm:N} 
    Assume (A1),(A2),(A3) and $\omega^2\geq 0$. Then~\eqref{eq:NLCurlCurlN} has a ground state and infinitely
    many bound states in $\mathcal V\oplus \mathcal W$.
  \end{thm}
  
  Almost the same proof gives the corresponding result for the  Dirichlet
  problem~\eqref{eq:NLCurlCurlD} and we shall comment on the necessary modifications in 
  Appendix~\ref{sec:Dirichlet}. In particular, the proof of Theorem~\ref{thm:N} indicates an alternative
  method to prove \cite[Theorem~2.2]{BarMed} under slightly different assumptions on
  the data. It is noteworthy that our proof, which relies on the dual variational method, avoids saddle-point 
  reductions to the Nehari-Pankov manifold and the rather involved critical point theory for strongly
  indefinite functionals.
   
%   The proof, however, is different given that we implement the same strategy as in the proof of Theorem~\ref{thm:N}, but now on the Hilbert space $\mathcal
%   H_0:=H_0(\curl;\Omega)$ that is introduced as the completion of $C_0^\infty(\Omega;\R^3)$ with respect to
%   the inner product $\skp{\cdot}{\cdot}$ from~\eqref{eq:innerproduct}. Define the subspaces
%    \begin{align*}      
%       \mathcal V_0 :=  \Big\{ E_1\in \mathcal H_0: \int_\Omega
%        \eps(x) E_1\cdot  \nabla \Phi\,dx = 0 \text{ for all } \Phi\in C_0^1(\Omega)\Big\}, \qquad 
%       \mathcal W_0^p := \big\{ \nabla u : u\in W_0^{1,p}(\Omega)\big\}. 
%   \end{align*} 
%   Again, a Sobolev Embedding Theorem for $\mathcal V_0$ can be proved, so $I$ is well-defined  on the space 
%   $\mathcal V_0\oplus \mathcal W_0^p$ provided that $2<p<6$.  A ground state of~\eqref{eq:NLCurlCurlD} is, by
%   definition, a critical point of $I$ over $\mathcal V_0\oplus \mathcal W_0^p$ that has least energy among all
%   nontrivial critical points.    
%    
%   \begin{thm} \label{thm:D} 
%     Assume (A1),(A2),(A3)  and $\omega^2\geq 0$.
%     Then~\eqref{eq:NLCurlCurlD} has a ground state and infinitely many bound states in
%     $\mathcal V_0\oplus \mathcal W_0^p$.
%   \end{thm}
%   
%   We shall present the reasoning in a very condensed way given that the arguments from the proof of
%   Theorem~\ref{thm:N} carry over almost verbatim.  
%     
  
  %\r{It will be convenient to write $m(x):=\eps(x)^{1/2}$ for some symmetric positive definite matrix
  %satisfying $m(x)\cdot m(x)=\eps(x)$.}
  
  \medskip
  
  We demonstrate that the dual variational approach is applicable on $\R^3$ as well. This problem is
  substantially different given that it resembles a nonlinear Helmholtz equation rather than an  
  elliptic boundary value problem. We have to make our assumptions on the
  permittivity $\eps$ and permeability $\mu$ more restrictive by requiring both to be constant and scalar: 
  $(\eps,\mu)\equiv (\eps_0,\mu)$ where $\eps_0\mu_0\in (0,\infty)$. This leads to the problem
  \begin{equation}\label{eq:NLCurlCurlR3}
    \nabla\times\nabla\times E - \omega^2\eps_0\mu_0 E = f(x,E)\qquad\text{in }\R^3.
  \end{equation}
  Our aim is to prove the existence of infinitely many $L^p$-solutions
  for this problem.  The main difference compared to the case of a bounded domain
   is that the curl-curl operator on $\R^3$ does not come with discrete point spectrum in $(0,\infty)$ but
  continuous spectrum just like the Helmholtz operator. In particular, a resolvent at the spectral parameter
  $\omega^2\eps_0\mu_0 >0$ does not exist, but some sort of right inverse can be constructed by means of the
  Limiting Absorption Principle.
  This linear operator $\mathcal R$, defined in~\eqref{eq:Romega2} below, enjoys
  boundedness and compactness properties as an operator from the divergence-free functions
  in $L^{p'}(\R^3;\R^3)$ to the divergence-free functions in $L^p(\R^3;\R^3)$ provided that $4<p<6$. 
  Evequoz and Weth~\cite{EveqWeth_Dual} showed how to exploit these properties 
  in the context of dual variational methods for nonlinear Helmholtz equations. In order to adapt
  this to the Maxwell setting on $\R^3$ we sharpen our assumptions on the nonlinearity.
   
   \begin{itemize}
     \item[(A4)] $f$ satisfies (A3) with $4<p<6$ where $f_0(\cdot,s)$ is $\Z^3$-periodic
    for all $s\in\R$ and there are $c,C>0$ with 
    $$
      cs^{p-2} \leq \partial_s f_0(x,s) \leq C s^{p-2} \qquad\text{for almost all }x\in\R^3 \text{ and all
      }s\in\R.
    $$
   \end{itemize}
   
   Our main result about~\eqref{eq:NLCurlCurlR3} reads as follows.
  
  \begin{thm}\label{thm:R3}
    Assume (A4) and $\omega^2\eps_0\mu_0\in (0,\infty)$.
    Then the equation~\eqref{eq:NLCurlCurlR3} admits a dual ground state and infinitely many
    geometrically distinct solutions in $L^{p}(\R^3;\R^3)$.
  \end{thm}
    
  Here, the notion of a dual ground state is the same as in \cite{EveqWeth_Dual}, i.e., $P:=f(x,E)$ is a ground state
  for some associated dual functional $J$ involving the operator $\mathcal R$, see Section~\ref{sec:TheoremR3}
  and in particular \eqref{eq:defJ_R3} further below. 
  As a new feature compared to the Nonlinear Helmholtz Equation \cite{EVEQ} our analysis makes use of the so-called
  Div-Curl-Lemma. 
  %In the proofs of our Theorems~\ref{thm:N}
  %and~\ref{thm:D} and~\ref{thm:R3} we do not consider the electric field $E$ as the unknown, but rather the
  %vector field $P:= \eps(x)^{-1}f(x,E)$.
  %Another remarkable difference concerns the nonexistence of ground states for \eqref{eq:NLCurlCurlR3} in the
  %case $\omega^2\eps_0\mu_0<0$. This is a consequence of the counterexamples
  %from~\cite{BaDoPlRe_GroundStates,Man_Nonlocal}, which are as relevant in a bounded domain
  %as in the full space.
  In view of the result on bounded domains, one may aim for an extension of Theorem~\ref{thm:R3} to more
  general permittivities $\eps$ and permeabilites $\mu$. Here the main challenge is the construction of a
  bounded linear operator  $\mathcal R$ with analogous properties as well as a corresponding Helmholtz
  Decomposition Theorem.
  %Related $L^p-L^q$-estimates  can be found in the literature, e.g., \cite{CosMan}.
  
  \medskip
  
  We emphasize that, up to our knowledge, this is the
  first variational existence result for \eqref{eq:NLCurlCurlR3} in the case $\omega^2\eps_0\mu_0>0$ and
  it is much stronger than Theorem~3(ii) in~\cite{Man_Uncountably} which has been obtained by a fixed point argument.  Note that   much
  research has been devoted to the complementary case $\omega^2\eps_0\mu_0\leq 0$ within the framework of
  cylindrically symmetric and thus divergence-free solutions
  \cite{AzzBenDApFor_Static,DApSic_Magnetostatic,BaDoPlRe_GroundStates,HirRei_Cylindrical,Bieganowski} where
  the variational analysis is somewhat parallel to the well-studied case of stationary nonlinear Schr\"odinger
  equations on $\R^3$. 
  This follows from the identity $\nabla\times\nabla\times E = -\Delta E$ for divergence-free vector fields
  $E$.  

\medskip

  Our strategy is to prove the above-mentioned results using dual variational methods based on a
  partially new variant of the Symmetric Mountain Pass Theorem (SMPT). Note that 
  the classical SMPT \cite[Corollary~2.9]{AmbRab} is not sufficient to prove Theorem~\ref{thm:R3} given that
  the Palais-Smale condition does not hold. 
  For this reason we  first provide a critical point theorem (Theorem~\ref{thm:CPTheorem}) that allows to
  prove both our main theorems simultaneously. In this result, the $C^1$-functional is only required to
  be ``PS-attracting'' in the sense of Definition~\ref{def:PSattracting} below. Having proved this
  result in Section~\ref{sec:SMPT}, we apply it in our proofs of Theorem~\ref{thm:N} in   
  Section~\ref{sec:TheoremN} and Theorem~\ref{thm:R3} in Section~\ref{sec:TheoremR3}.
  Section~3 provides the Linear Theory needed for the analysis in Section~\ref{sec:TheoremN}.
  
%   
%   \begin{center}
%      \textbf{Overview how Theorem~\ref{thm:CPTheorem} is applied} \\ \smallskip    
%      \renewcommand{\arraystretch}{1.5}
%      \begin{tabular}{|c|c|c|c|c|}  \hline%\smallskip   
%     \textbf{Frequencies} & $\Omega$ & \textbf{Theorem} &
%     \textbf{Dual Functional}   \\  \hline\hline 
%     $\omega^2\in(0,\infty)\sm\sigma(\cL)$ & bounded $C^1$-domain &  Theorem~\ref{thm:N} &
%     \eqref{eq:functionalJdualN}  \\
%     $\omega^2\in\sigma(\cL)$ & bounded $C^1$-domain &  Theorem~\ref{thm:N} &
%     \eqref{eq:functionalJdualNomega2}  \\
%     $\omega^2=0$ & bounded $C^1$-domain &  Theorem~\ref{thm:N} & \eqref{eq:functionalJdualN0}  \\
%     $\omega^2\in (0,\infty)$ & $\R^3$  &  Theorem~\ref{thm:R3} &   \eqref{eq:defJ_R3}  \\
%     \hline 
%   \end{tabular} 
%   \end{center}
%   
%   \r{Include solution space: Clarify $\mathcal V$ and $\mathcal W$ according to all cases}
%   
%   \medskip  

\section{A Symmetric Mountain Pass Theorem without PS-condition} \label{sec:SMPT}

  In this section we prove a variant of the Symmetric Mountain Pass Theorem with the distinguishing
  feature that it does not require the Palais-Smale condition. The idea for this abstract result and its proof
  is due to Szulkin-Weth \cite{SzuWet} and an earlier paper by Bartsch-Ding
  \cite{BarDing_NLSperiodic} where the existence of infinitely many geometrically distinct solutions has been proved for some periodic
  nonlinear Schr\"odinger equation of the form 
  $$
    -\Delta u + V(x)u = f(x,u) \qquad\text{in }\R^n
  $$
  with an odd nonlinearity $f(x,\cdot)$. Several refinements and applications can be found in the
  literature. Here we want to highlight the contributions by Squassina and Szulkin~\cite{SquSzu} and 
  Evequoz~\cite{EVEQ} given their relevance for our paper. Our first goal is provide an abstract critical
  point theorem for functionals $J:Z\to\R$ on a general Banach space with some sort of Mountain
  Pass Geometry. In general, the PS-condition for $J$ may fail, but some useful but weaker compactness property, a 
  ``PS-attracting'' property,  is required to hold.
  This property is essentially extracted from the paper by Szulkin and Weth \cite{SzuWet} and functionals
  satisfying the PS-condition are easily seen to be PS-attracting, see Remark~\ref{rem:Abstract}(a). As a
  consequence, our variant of the SMPT   generalizes the original one and yields the existence
  of infinitely many solutions in various applications. In the presence of discrete translational equivariance
  it provides the existence of infinitely many geometrically distinct solutions, see
  Remark~\ref{rem:Abstract}(b).
  
  \medskip
    
  For a given functional $J\in C^1(Z)$ we define the sets of critical points  
  $$
    \mathcal{K}:=\left\{u \in Z: J^{\prime}(u)=0\right\},\qquad
    \mathcal{K}_d:=\left\{u \in \cK: J(u)=d\right\}.
  $$
  We first introduce the notion of a $\cK$-decomposition where $\cK$ is subdivided into suitable bounded and
  symmetric pieces that are discrete. Here, a set $A\subset Z$ is called symmetric
  provided that it is even, i.e., we have $x\in A$ if and only if $-x\in A$.

\begin{defn} \label{defn:Kdecomposition}
  Let $J\in C^1(Z)$ be even. We say that $J$ admits a $\mathcal K$-decomposition $\cK=\bigcup_{i\in I} \cK^i$ 
  for some set $I$ if the following holds for all $i,j\in I$:
  \begin{itemize}
    \item[(i)] $\cK^i$ is bounded, symmetric and non-empty, 
    \item[(ii)] $\inf \{ \|u-v\|: u,v\in \cK^i\cup \cK^j, u\neq v\}>0$,
    \item[(iii)] $\{J(u):u\in \cK^i\}$ is a finite set.
  \end{itemize}
  Such a $\cK$-decomposition is called finite if $\#I<\infty$ and infinite if $\#I=\infty$.   
\end{defn}

We stress that (ii) needs to be checked for $i=j$ as well.
  To formulate our result we introduce the notion of a $PS$-attracting functional.

\begin{defn}\label{def:PSattracting}
  We say that $J$ is $PS$-attracting if for any given Palais-Smale sequences $(v_n)_n,(w_n)_n$ of $J$
  we either have $\|v_n -w_n \| \to 0$ as $n \to \infty$ or 
  $$
    \limsup_{n \to \infty}\|v_n-w_n \| \geq \kappa
    \quad\text{where }\kappa:= \inf\big\{\|u-v\|: u,v\in\cK, u\neq v\big\}.  
  $$
\end{defn}

We stress that there is no need to prove $\kappa>0$. The importance of this observation is discussed in
Remark~\ref{rem:Abstract}(c). In fact, even in situations where the Palais-Smale condition holds, $\kappa$ can
be $0$. For instance, consider any functional $J\in C^1(\R)$ satisfying  $J(x)=1$ for $1\leq |x|\leq 2$ as well as $|J'(x)|\geq c>0$ outside some compact
set. Obviously, $J$ satisfies the PS-condition, but the set of critical points is not discrete and hence
$\kappa$ must be zero. The definition of $\kappa$ involves the set $\cK$ that is, of course, unknown a
priori. Nevertheless, we shall see below that this compactness property can be verified in applications. The
assumptions on   the energy functional are the following:
\begin{itemize}
  \item[($G_1$)] $Z$ is a Banach space $Z$ with $Z=Z^+\oplus Z^-$ and 
  $\dim(Z^-)<\infty, \dim(Z^+)=\infty$.
  \item[($G_2$)] $J\in C^1(Z)$ is even with $J(0)=0$ and, for some $\rho>0$, 
  $$
    \inf_{S_\rho^+} J>0
    \qquad\text{where}\qquad S_\rho^+ = \{v\in Z^+: \|v\|=\rho\}.
  $$
  \item[($G_3$)] For any given $m\in\N$ there is a finite-dimensional subspace $Z_m\subset Z^+$ such that
  $J(u)\to -\infty$ uniformly for $u\in
  Z^-\oplus Z_m$ as $\|u\|\to\infty$ and $\dim(Z_m)\nearrow +\infty$ as $m\to\infty$.
  \item[($G_4$)] $J$ is PS-attracting.
\end{itemize}
Note that $Z^-=\{0\}$ is a valid choice in $(G_1)$ and the uniformity in $(G_3)$ need only  hold
for a fixed $m\in\N$. Our variant of the Symmetric Mountain Pass Theorem without   Palais-Smale condition   reads
as follows.

\begin{thm}\label{thm:CPTheorem}
  Assume $(G_1),(G_2),(G_3),(G_4)$. Then every $\cK$-decomposition of $J$ is infinite. In particular, $J$ has
  infinitely many pairs of critical points.
\end{thm}

\begin{rem} \label{rem:Abstract} ~
\begin{itemize}
  \item[(a)] If $J\in C^1(Z)$ satisfies the  PS-condition, then $J$ is
  $PS$-attracting. Indeed, for two PS-sequences $(v_n),(w_n)$ with $\limsup_{n\to\infty} \|v_n-w_n\|\geq
  \mu>0$ we may subsequently pass to  subsequences such that $(v_n),(w_n)$ strongly converge to $v,w$, respectively. 
  The continuity of $J'$ implies that $v,w$ are critical points of $J$ with $\|v-w\|\geq \mu>0$. So $v\neq w$
  and  
  $$
    \limsup_{n\to\infty} \|v_n-w_n\| \geq \|v-w\| \geq \kappa.
  $$
  So $J$ is $PS$-attracting. This implication shows that 
  Theorem~\ref{thm:CPTheorem} contains the original Symmetric Mountain
  Pass Theorem  \cite[Corollary~2.9]{AmbRab} and even the more general
  version \cite[Theorem~6.5]{Struwe} as special cases.  
  \item[(b)] In the context $\Z^d$-equivariant stationary nonlinear Schr\"odinger equations this theorem
  applies to $\cK$-decompositions given by sets of the form 
  $$
    \cK^v = \{v(\cdot+n):n\in\Z^d\}\cup \{-v(\cdot+n):n\in\Z^d\}.
  $$ 
  Indeed, in the papers \cite{SquSzu,EVEQ} the authors implicitly check the properties required in
  Definition~\ref{defn:Kdecomposition}, notably that different orbits have positive distance to each other,
  see \cite[Lemma~3.1]{EVEQ}.    
  The PS-attracting property of the energy functional is verified using the Nonvanishing Property 
  \cite[Theorem~2.3]{EveqWeth_Dual} of $\mathcal R$ as well as the Rellich-Kondrachov theorem. Our
  reasoning in the proof of Theorem~\ref{thm:R3} follows the same lines.
  % The Theorem implies that every $\cK$-decomposition is infinite, and in particular infinitely
  %many Orbits $\cK^v$ containing critical points. It was proved in \r{[Evequoz: Lemma~3.2]} that the energy
  %functional is PS-attracting.
  \item[(c)] Theorem~4.2 in \cite{BarDing_NLSperiodic} is more general regarding the regularity assumptions on
  the functional and it even allows for infinite-dimensional $Z^-$ up to modifications of the topology on
  $Z^-$.
  On the other hand, in assumption ($\Phi_5$) the authors assume $\kappa>0$.   
  This is responsible for the fact that this Theorem cannot be applied directly in typical situations, but must be
  used in some indirect reasoning just as in the proof of Theorem~1.2 in~\cite{BarDing_NLSperiodic}.
  In particular, \cite[Theorem~4.2]{BarDing_NLSperiodic} does not generalize the SMPT. 
  Theorem~\ref{thm:CPTheorem} removes this inconvenience. 
\end{itemize}
\end{rem}

 The strategy to prove Theorem~\ref{thm:CPTheorem} is as follows: We assume for contradiction that the
 claim is false, i.e., 
\begin{equation}\label{eq:Assumption}
  \text{There is a finite }\cK-\text{decomposition for }J.
\end{equation}   
With this assumption and Definition~\ref{defn:Kdecomposition}(ii) we know that $\kappa$ from
Definition~\ref{def:PSattracting} is positive. This enables us to perform a deformation argument
involving the minmax values 
$$
  d_k := \inf_{A\in\Sigma, \iota^*(A)\geq k} \sup_{u\in A} J(u).
$$
where $\Sigma := \{A\subset Z: A \text{ is symmetric and compact}\}$ and $\iota^*$ is defined
as in~\cite{SquSzu,EVEQ}, namely
\begin{align*}
  \iota^*(A) &:= \min\big\{\gamma(h(A)\cap S_{\rho}^+):\;  h\in \cH\big\} \qquad \text{where}\\
  \cH&:= \big\{h:Z\to h(Z)\text{ is an odd homeomorphism with }J(h(u))\leq J(u) \text{ for all }u\in
Z\big\}  
\end{align*}
with $\rho$ as in $(G_2)$ and the Krasnoselski genus $\gamma$. 
Exploiting \eqref{eq:Assumption} and hence $\kappa>0$ we show that the $d_k$ are critical values of $J$ and
form an increasing sequence. In particular, infinitely many different critical values exist. This, however,
contradicts the assumption~\eqref{eq:Assumption} given that Definition~\ref{defn:Kdecomposition}(iii)
implies that finite $\cK$-decomposition   allow for at most finitely many critical values. This contradiction
finishes the proof.

\medskip

\begin{prop} \label{prop:dk}
   Given the assumptions of Theorem~\ref{thm:CPTheorem} we have  $0< d_k\leq d_{k+1}<\infty$ for
   all $k\in\N$.
\end{prop}
\begin{proof}
  The inequality $d_k\leq d_{k+1}$ holds by definition. These values are not $+\infty$ given that, for any
  given $k\in\N$, the combination of $(G_1),(G_2),(G_3)$ and the argument from \cite[Lemma~2.16 (iv)]{SquSzu}
  or \cite[Lemma~4.5]{BarDing_NLSperiodic} 
  implies the existence of a set $A\in\Sigma$ with $\iota^*(A)\geq k$. 
  Moreover, any $A\in \Sigma$
  with $\iota^*(A)\geq  1$ implies $\gamma(h(A)\cap S_\rho^+) \geq 1$ for some $h\in\cH$, in particular
  $h(A)\cap S_{\rho}^+\neq\emptyset$. So
  $$
    d_k\geq d_1 
    = \inf_{A\in\Sigma, \iota^*(A)\geq 1} \sup_{u\in A} J(u) 
    \geq \inf_{h(A)\cap S_\rho^+ \neq \emptyset} \sup_{u\in A} J(h(u)) 
    \geq \inf_{S_\rho^+} J > 0.
  $$ 
\end{proof}

In the following, we write $U_\delta(\cK_d):=\{z\in Z: \dist(z,\cK_d)<\delta\}$, so
$U_\delta(\emptyset)=\emptyset$. 

\begin{prop} \label{prop:preliminaries}
  In addition to the assumptions of Theorem~\ref{thm:CPTheorem} suppose \eqref{eq:Assumption}. Let $\kappa$
  be given as in Definition~\ref{def:PSattracting}.
\begin{itemize}
%   \item[(b)] 
%   $$
%     \kappa:=\inf \left\{\|v-w\|_Z: v, w \in \mathcal{K}, v \neq w\right\}>0
%   $$
  \item[(a)] For $0<\delta<\kappa$ and $d\in\R$ such that for all non-empty $K\subset \cK_d$ we
  have $\tau>0$ where
  \begin{align*}
    \tau := \inf \left\{\|J'(v)\|: v \in U_\delta(K) \sm U_{\frac{\delta}{2}}(K)\right\}.     
  \end{align*}   
  \item[(b)] If $d\in\R\sm\{0\}$ then $\gamma(\cK_d)\leq 1$. 
  \item[(c)] For all $d\in\R$ there is $\eps_0>0$ such that 
  $\cK_{\tilde d}=\emptyset$ for $0<|d-\tilde d|<\eps_0$.
\end{itemize}
\end{prop}
\begin{proof}
  To prove (a) assume for contradiction that there is a sequence $(v_n)_n \subset U_\delta(K)
  \sm U_{\delta/2}(K)$ such that $J'(v_n) \to 0$. Then we can find a sequence
  $(w_n)\subset K\subset \cK_d$ with 
  $$ 
    0<\frac{\delta}{2}\leq \|v_n-w_n\|\leq \delta<\kappa\quad\text{for all }n\in\N.
  $$ 
  Since $(v_n),(w_n)$ are PS-sequences and $J$ is PS-attracting, this contradicts our choice of
  $\kappa$. So we must have $\tau>0$ and (a) is proved. 
  Since $J$ is $PS$-attracting and $d\neq 0$,  $\cK_d$
  is a discrete set that does not contain $0$.  Hence, $\cK_d = \bigcup_{i\in I} \{v_i,-v_i\}$ is a disjoint
  union for some set $I$, then $h:\cK_d\to\R\sm\{0\}, \pm v_i\mapsto \pm 1$ defines an odd and continuous map, 
  so $\gamma(\cK_d)\leq 1$. This proves (b).
  %$\cK_d = \bigcup_{[v]\in \cK_d/\sim} \{v,-v\}$ $h:\cK_d/\sim\to \R, \pm v \mapsto \pm 1$
  Finally,  \eqref{eq:Assumption} and Definition~\ref{def:PSattracting}(iii) imply that the set of critical
  values $\{J(u): u\in\cK\}$ is finite and hence discrete. This gives (c).
\end{proof}
 
% 
% \begin{lem} \label{lem:DiscretePSsequences}
%  Assume (AA). Let  $\left(v_n^1\right)_n,\left(v_n^2\right)_n$ two (automatically bounded)  
%  Palais-Smale sequences for $J$, then either $\left\|v_n^1-v_n^2\right\| \rightarrow 0$ as $n \rightarrow
%  \infty$ or $\limsup _{n \rightarrow \infty}\left\|v_n^1-v_n^2\right\| \geq \kappa$.
% \end{lem}
% \b{Proof: BLACKBOX} 

To prove the theorem we consider a pseudo-gradient field on $Z\sm \cK$, i.e., a 
locally Lipschitz continuous map $H: Z\sm \cK \rightarrow Z$ 
%(see for instance Lemma~5.1.8 in [Chang: Methods inNonlinear Analysis])\r{general Banach spaces?} 
with  
\begin{equation}\label{eq:pseudogradientI}
  \|H(w)\|\leq 2\|J'(w)\|,\qquad J'(w)[H(w)] \geq \|J'(w)\|^2. 
\end{equation}
We define the associated flow $\eta$ as the unique maximal solution of the initial value problem 
\begin{equation}\label{eq:pseudogradientII}
  \frac{d}{d t} \eta(t, w)=- H(\eta(t, w)),%\chi(J(\eta(t,w))),\qquad  
  \qquad \eta(0, w)=w
  \qquad\text{for }w\in Z\sm \cK. 
\end{equation}
The maximal existence time in positive time direction is denoted by $T^{+}(w)$. The
flow $\eta$ is continuous at all points $(t,w)$ with $w\in Z\sm\cK$ and $0\leq t<T^+(w)$ by unique
local solvability of \eqref{eq:pseudogradientII}. The most important property is that $t\mapsto
J(\eta(t,w))$ is decreasing on $[0,T^+(w))$ due to
$$
  \frac{d}{dt}\Big(J(\eta(t,w))\Big)
  = -J'(\eta(t,w))[H(\eta(t,w))]
  \leq - \|J'(\eta(t,w))\|^2 
  \qquad\text{for }0<t<T^+(w).
$$
We use this in the following deformation argument involving suitable sublevel sets $J^c:=\{u\in Z:J(u)\leq
c\}$.

\begin{lem} \label{lem:entrancetime}
  In addition to the assumptions of Theorem~\ref{thm:CPTheorem} suppose \eqref{eq:Assumption}, let $d:=d_k$ for some
  $k\in\N$ and $0<\delta<\kappa$. Then there exists $\varepsilon>0$ such that  
    \begin{equation} \label{eq:deformationI}
      \lim _{t \rightarrow T^{+}(w)} J(\eta(t, w))<d-\varepsilon
      \quad\text{for all } w \in J^{ d+\varepsilon} \sm U_\delta(\cK_d).
    \end{equation} 
   Moreover, the entrance time map
    $$
      e: J^{ d+\varepsilon} \sm  U_\delta(\cK_d)  \to [0,\infty),\quad  w\mapsto \min\big\{ t\geq 0:
      J(\eta(t, w))\leq d-\varepsilon\big\}
    $$
    is well-defined, even and continuous.  
\end{lem}
\begin{proof} 
  In the following we verify the claim for 
  $$
    0<\eps< \min\left\{ \frac{\delta\tau}{\sqrt{32}}, \frac{\kappa}{2},\eps_0\right\}
  $$
  where $\tau,\kappa,\eps_0>0$ are as in Proposition~\ref{prop:preliminaries}. Assume for contradiction that 
  \eqref{eq:deformationI} does not hold for such $\eps$. Then we find some $w \in J^{d+\varepsilon} \sm
  \ov{U_\delta(\cK_d)}$ such that
  \begin{equation} \label{eq:entrancetimeI}
    d-\eps\leq J(\eta(t,w))\leq d+\eps\quad\text{for all }t\in [0,T^+(w)).
  \end{equation}
  Using \eqref{eq:entrancetimeI} we first show that the flow converges to some critical
  point at energy level $d$, i.e.,
  \begin{equation} \label{eq:deformationII}
    \lim_{t\to T^+(w)} \eta(t,w) = w^* \qquad\text{where }w^*\in\cK_d.
  \end{equation} 
  We first prove this in the case $T^{+}(w)<\infty$. For $0<\tilde t<t<T^{+}(w)$ we have  
\begin{align} \label{eq:deformationIII}
\begin{aligned}
\|\eta(t, w)-\eta(\tilde t, w)\| 
& \leq  \int_{\tilde t}^t\| H(\eta(s, w))\| \,ds 
\;\stackrel{\eqref{eq:pseudogradientI}}\leq\; 2 \int_{\tilde t}^t\| J'(\eta(s, w))\| \,ds \\ 
 &\stackrel{\eqref{eq:pseudogradientI}}\leq 2  \int_{\tilde t}^t \sqrt{J'(\eta(s, w))[ H(\eta(s, w))]} \,ds \\
& \leq 2 \sqrt{t-\tilde t}\left(\int_{\tilde t}^t J'(\eta(s, w))[ H(\eta(s, w))]
\,ds\right)^{\frac{1}{2}} \\ 
 &\stackrel{\eqref{eq:pseudogradientII}}=2 \sqrt{t-\tilde t}\,\Big(J(\eta(\tilde t, w))-J(\eta(t,
 w))\Big)^{\frac{1}{2}} \\
&\stackrel{\eqref{eq:entrancetimeI}}\leq 2 \sqrt{t-\tilde t}\, (2\eps)^{\frac{1}{2}}.
\end{aligned}
\end{align}
 Since $T^+(w)$ is finite, this estimate implies that 
 $\eta(t,w)$ is Cauchy and hence converges as $t\to T^+(w)$. The limit $w^*$ must be a critical point of $J$
 because otherwise the trajectory $t \mapsto \eta(t, w)$ could be continued beyond $T^{+}(w)$ thanks to the
 local solvability of the initial value problem~\eqref{eq:pseudogradientII} with initial data in $Z\sm \cK$. 
 So we have $w^*\in\cK$ and \eqref{eq:entrancetimeI} gives $d-\eps\leq J(w^*)\leq d+\eps$. From
 $0<\eps<\eps_0$ and Proposition~\ref{prop:preliminaries}(c) we conclude $w^*\in\cK_d$, so
 \eqref{eq:deformationII} is proved.  \\
 Next we show \eqref{eq:deformationII} in the case  $T^{+}(w)=\infty$. Since the map $t\mapsto J(\eta(t,w))$
 is decreasing and bounded from below in view of \eqref{eq:entrancetimeI}, we find that $J(\eta(t,w))$
 converges as $t\to +\infty$. We now deduce that $\eta(t,w)$   converges as well. 
 Assume for contradiction that this is not the case. Then there exists a sequence $\left(t_n\right)_n
 \subset[0, \infty)$ with $t_n \rightarrow \infty$ and $\| \eta\left(t_n, w\right)-$ $\eta\left(t_{n+1},
 w\right) \|=\varepsilon$ for every $n$. Choose the smallest $t_n^1 \in\left(t_n, t_{n+1}\right)$ such that
 $\| \eta\left(t_n, w\right)- \eta\left(t_n^1, w\right) \|=\frac{\varepsilon}{3}$ and let 
$\mu_n:=\min _{s \in\left[t_n, t_n^1\right]}\| J'(\eta(s, w))\|$. Then
$$
\begin{aligned}
\frac{\varepsilon}{3} 
& =\left\|\eta\left(t_n^1, w\right)-\eta\left(t_n, w\right)\right\| 
\leq \int_{t_n}^{t_n^1}\| H(\eta(s, w)) \| d s
\stackrel{\eqref{eq:pseudogradientI}}\leq 2 \int_{t_n}^{t_n^1}\|J'(\eta(s, w))\| d s \\
& \leq \frac{2}{\mu_n} \int_{t_n}^{t_n^1}\|J'(\eta(s, w))\|^2 d s 
\stackrel{\eqref{eq:pseudogradientI}}\leq \frac{2}{\mu_n}
\int_{t_n}^{t_n^1} J'(\eta(s,w))[H(\eta(s, w))] d s \\
%&= -\frac{4}{\kappa_n}
%\int_{t_n}^{t_n^1} J'(\eta(s,w))[\frac{d}{ds} \eta(s, w)] d s \\
&\stackrel{\eqref{eq:pseudogradientII}}=\frac{2}{\mu_n}\Big(J\left(\eta\left(t_n,
w\right)\right)-J\left(\eta\left(t_n^1, w\right)\right)\Big).
\end{aligned}
$$
Since $J(\eta(t,w))$ converges as $t\to \infty$, we have $J(\eta(t_n,
w))-J(\eta(t_n^1, w)) \rightarrow 0$ and thus $\mu_n \rightarrow 0$ as
$n\to\infty$.
So there exist $s_n^1 \in\left[t_n, t_n^1\right]$ such that $ J'(v_n) \rightarrow 0$, where
$v_n:=\eta(s_n^1, w)$.   Similarly, we find a largest $t_n^2 \in(t_n^1, t_{n+1})$ for which
$\left\|\eta(t_{n+1}, w)-\eta(t_n^2, w)\right\|=\frac{\varepsilon}{3}$ and then $w_n:=\eta(s_n^2, w)$
satisfies $J'(w_n) \rightarrow 0$. As $\left\|v_n-\eta(t_n, w)\right\| \leq \frac{\varepsilon}{3}$ and
$\left\|w_n-\eta(t_{n+1}, w)\right\| \leq \frac{\varepsilon}{3}$,
 $$ 
   \left(v_n\right)_n,\left(w_n\right)_n \quad\text{are Palais-Smale sequences with}\quad 
  \frac{\varepsilon}{3} \leq\left\| v_n-w_n \right\| \leq 2 \varepsilon<\kappa.
$$
This, however, contradicts our choice of $\kappa$, so the
assumption was false.
Hence, $\eta(t, w)$ converges as $t\to\infty$. The limit must be a critical point of $J$ because
 otherwise $\frac{d}{dt} \big(J(\eta(t,w))\big)$ would be uniformly negative for large $t$, which
 violates~\eqref{eq:entrancetimeI}. So~\eqref{eq:deformationII} also holds in the case $T^+(w)=+\infty$. 
  
  \medskip
  
% $\eta(t, w) \rightarrow w^*$ as $t \rightarrow T^{+}(w)$ for some critical point 
% with $d-\eps\leq J(w^*)\leq d+\eps$. In view of (a), this gives $J(w^*)=d$.
  
  From \eqref{eq:deformationII} we infer that the flow $t\mapsto \eta(t,w)$ 
  eventually enters the region $U_\delta(\{w^*\})$ at some time $t_1\in (0,T^+(w))$ and remains outside
  of the region $U_{\delta/2}(\{w^*\})$ until some time $t_2\in (t_1,T^+(w))$. Formally,
  \begin{align*}
   t_1 &:=\max \big\{t \in\left[0, T^{+}(w)\right): \eta(t, w) \notin U_\delta(\{w^*\})\big\},  \\
   t_2 &:=\inf \big\{t \in\left(t_1, T^{+}(w)\right): \eta(t, w) \in U_{\delta/2}(\{w^*\})\big\}.
 \end{align*} 
 As in \eqref{eq:deformationIII} we get the inequality 
\begin{align*}
  \frac{\delta}{2} \leq \|\eta(t_2, w)-\eta(t_1, w)\| \leq 2 \sqrt{t_2-t_1}\,(2\eps)^{1/2},
  \quad\text{and thus } t_2-t_1 \geq \frac{\delta^2}{32\eps}.
\end{align*}
On the other hand, Proposition~\ref{prop:preliminaries}(a) gives  $\|J'(\eta(s,w)\|\geq \tau>0$  for all 
$s\in [t_1,t_2]$ and thus
\begin{align*}
d 
&\stackrel{\eqref{eq:deformationII}}= \lim_{t\to T^+(w)} J\left(\eta\left(t, w\right)\right)   
\;\leq\; J\left(\eta\left(t_2, w\right)\right)  \\
&\stackrel{\eqref{eq:entrancetimeI}}\leq d+\eps+ J\left(\eta\left(t_2, w\right)\right)-J\left(\eta\left(t_1,
w\right)\right) \\
&\stackrel{\eqref{eq:pseudogradientII}}= d+\eps-\int_{t_1}^{t_2} J'(\eta(s, w))[H(\eta(s, w))]\,ds \\
&\stackrel{\eqref{eq:pseudogradientI}}\leq d+\eps-  \int_{t_1}^{t_2}\|J'(\eta(s, w))\|^2 \,ds \\ 
&\leq d+\eps- \tau^2\left(t_2-t_1\right) 
\;\leq\; d+\eps-\frac{\tau^2\delta^2}{32\eps} \\
&< d,
\end{align*}
which yields a contradiction. So \eqref{eq:entrancetimeI} is false and \eqref{eq:deformationI} is proved.
 
 \medskip

Finally, thanks to~\eqref{eq:deformationI} we know that $e(w)\in [0,T(w))$ is a well-defined finite
nonnegative number characterized by $J(\eta(e(w),w))= \min\{J(w),d-\eps\}$. Since $J$ and $\eta$ are
continuous and $t\mapsto J(\eta(t,w))$ is decreasing near each $w\in J^{d+\eps}\sm U_{\delta}(\cK_d)$, we
obtain that the entrance time map is continuous.
% Here we
% exploited that $t\mapsto J(\eta(t,w_{n_j}))$ is nonincreasing.
% So we infer $e+\gamma\geq e(w)$ and, since $\gamma>0$ was arbitrary, we conclude $e\geq e(w)$. If
% $e>e(w)$, then $e(w_n)>e(w)$ for almost all $n\in\N$ and
% \begin{align*}
%    J(\eta(e(w),w))
%   &= \lim_{j\to\infty} J(\eta(e(w),w_{n_j})) \\
%   &= \lim_{j\to\infty} J(\eta(e(w_{n_j}),w_{n_j})) + \int_{e(w)}^{e(w_{n_j})}
%   J'(\eta(s,w_{n_j}))[H(\eta(s,w_{n_j}))] \,ds  \\
%   &\geq d-\eps + \liminf_{j\to\infty} \int_{e(w)}^{e(w_{n_j})} \|J'(\eta(s,w_{n_j}))\|^2 \,ds \\ 
%   &\geq d-\eps + \tau^2 \liminf_{j\to\infty} (e(w_{n_j})-e(w)) \\
%   &= d-\eps + \tau^2(e-e(w)),  
% \end{align*}
% which contradicts $J(\eta(e(w),w))=d-\eps$. So we must have $e\leq e(w)$ as well, hence $e=e(w)$. So $e(w)$ is
% the only accumulation point of $(e(w_n))$, which proves $e(w_n)\to e(w)$ as $n\to\infty$.
\end{proof}

\medskip

\textbf{Proof of Theorem~\ref{thm:CPTheorem}:}\; Assume for contradiction that \eqref{eq:Assumption} is true.
We then show  
$$
  \cK_{d_k}\neq \emptyset \qquad\text{and}\qquad  d_{k+1}>d_k  \quad\text{for all }k\in\N.
$$ 
Indeed, by the continuity property of the genus and Proposition~\ref{prop:preliminaries}(b), there exists
$\delta>0$ such that $\gamma(\bar{U})=\gamma(\cK_{d_k})\leq 1$ where
$U:=U_\delta(\cK_{d_k})$ and $0<\delta<\kappa$. We may then choose $\varepsilon>0$ as in
Lemma~\ref{lem:entrancetime}.
By definition of $\iota^*$ we can find $A\in\Sigma$ with $\iota^*(A)\geq k$ and $\sup_A J \leq d_k +\eps$. 
Then we have  $A\sm U\in\Sigma$ and, by~\eqref{eq:deformationI}, 
\begin{equation} \label{eq:deformationIV}
  \sup_{u\in A\sm U} J(\eta(e(u),u))\leq d_k-\eps.
\end{equation} 
The map $\tilde h(u):= \eta(e(u),u)$ is well-defined, odd and continuous on $J^{d_k+\eps}\sm U$ and can be
extended to some function $h\in\mathcal H$ that is defined on $J^{d_k+\eps}$. For instance,  define for 
$u\in J^{d_k+\eps}$
\begin{align*}
  h(u) &:= \eta(e(u),u) &&\hspace{-2cm}\text{if }u\in \dist(u,\cK_{d_k})\geq \delta,\\
  h(u) &:= \eta\Big( (2\lambda\delta^{-1}-1) e(u),u\Big) &&\hspace{-2cm}\text{if }\dist(u,\cK_{d_k})=\lambda
  \in \big( \frac{\delta}{2},\delta\big),\\ 
  h(u) &:= u  &&\hspace{-2cm}\text{if }\dist(u,\cK_{d_k})\leq \frac{\delta}{2}.
\end{align*}
From the definition of $d_l$ and \eqref{eq:deformationIV} we get $\iota^*(h(A\sm U))=\iota^*(\tilde h(A\sm
U))\leq k-1$.  So the mapping properties of $\iota^*$ from \cite[Lemma~2.16]{SquSzu} and
Proposition~\ref{prop:preliminaries}(b) give\footnote{Lemma 2.16 in \cite{SquSzu} is stated and proved 
only for $Z^-=\{0\}$, but the same proof yields the result for any finite-dimensional $Z^-$.
}
\begin{align*}
  k\leq \iota^*(A) 
  \leq \iota^*(A\sm U) + \gamma(\ov U)  
  \leq \iota^*(h(A\sm U)) + \gamma(\cK_{d_k})    
  \leq k-1 + \gamma(\cK_{d_k}),
\end{align*}
which implies  $\cK_{d_k}\neq \emptyset$. Moreover, $d_k= d_{k+1}$ is impossible. 
Indeed, under this assumption we may even choose  $A\in\Sigma$ with $\iota^*(A)\geq k+1$ with the properties
 given above, so 
 $$
   k+1 \leq \iota^*(A)\leq k-1+\gamma(\cK_d) \leq k,
 $$
 a contradiction. Given that the critical levels are nondecreasing, we obtain $d_{k+1}>d_k$ for all
 $k\in\N$. So there are infinitely many critical levels and every $\cK$-decomposition of $J$ is infinite. So
 the assumption \eqref{eq:Assumption} is false, which is all we had to prove.
 \qed

\section{The linear Neumann problem om bounded domains}

  We are first interested in a solution theory for the linear boundary value problem   
 \begin{equation} \label{eq:LinBVPN}
      \nabla\times \big(\mu(x)^{-1}\nabla\times E\big) -\lambda\eps(x) E 
      = \eps(x) g \quad\text{in }\Omega,\qquad\quad 
      (\mu(x)^{-1}\nabla\times E)\times\nu = 0 \quad\text{on }\partial\Omega
 \end{equation}
 where $\lambda\in\R$ and $E\in \mathcal V$. A 
 weak solution $E\in \mathcal V$ of \eqref{eq:LinBVPN} is characterized by 
\begin{equation} \label{eq:weaksolutionI} 
  \int_\Omega   \mu(x)^{-1}(\nabla\times E)\cdot(\nabla\times \Phi)\,dx
  -  \lambda\int_\Omega  \eps(x)E\cdot\Phi\,dx
  = \int_\Omega \eps(x)g\cdot\Phi\,dx 
  \qquad\text{for all }\Phi\in\mathcal V. 
\end{equation} 
In order to identify the right function space for $g$ we use a Helmholtz Decomposition. Define
  \begin{align*}
    X^{p'}
    &:= \Big\{ E\in L^{p'}(\Omega;\R^3): \int_\Omega \eps(x) E\cdot \nabla\phi\,dx = 0 \text{ for all
    }\phi\in W^{1,p}(\Omega)\Big\},\\
    Y^{p'} 
    &:= \big\{\nabla u: u\in W^{1,p'}(\Omega)\big\}.     
  \end{align*}
  By the Gauss-Green formula elements of $X^{p'}$ are not only divergence-free in $\Omega$, but also
  satisfy $\eps(x)E\cdot\nu= 0$ on $\partial\Omega$ in a distributional sense. We also need 
  \begin{align*}
    (Y^{p'})^{\perp_\eps} &:= \Big\{f\in L^{p}(\Omega;\R^3): \int_\Omega \eps(x)f\cdot g\,dx = 0 \text{ for all }g\in
    Y^{p'}\Big\},  \\
    (X^{p'})^{\perp_\eps} &:= \Big\{g\in L^{p}(\Omega;\R^3): \int_\Omega \eps(x)f\cdot g\,dx = 0 \text{ for all }f\in
    X^{p'}\Big\}.
  \end{align*} 
  From~\cite[Theorem~1]{AuscherQafsaoui} and $W^{1,3}(\Omega)\subset \text{VMO}(\Omega)$ 
  \cite[Theorem 3.3(ii)]{CiaPick} we get the following.
     
 \begin{prop}[Auscher, Qafsaoui] \label{prop:AuschQafI}
   Assume (A1),(A2) and $1<p<\infty$. Then for any given $f\in L^{p'}(\Omega;\R^3)$ the boundary value
   problem
   \begin{equation} \label{eq:BVPAuschQafI}
     \nabla\cdot \big(\eps(x)(f+\nabla u)\big) = 0 \quad\text{in }\Omega,\qquad 
     \eps(x)(f+\nabla u)\cdot \nu = 0 \quad\text{on }\partial\Omega
   \end{equation}
   has a weak solution in $u\in W^{1,p'}(\Omega)$ that is unique up to constants and satisfies 
   $\|\nabla u\|_{p'} \les \|f\|_{p'}$. 
 \end{prop} 
 
 This preliminary result provides a Helmholtz Decomposition that allows to set up a Fredholm theory for
 \eqref{eq:LinBVPN}. This problem reads $(\mathcal L-\lambda)E = g$ where  
  $$
    \mathcal L E := \eps(x)^{-1}\nabla\times \big(\mu(x)^{-1}\nabla\times E\big),\qquad
    \skp{E}{F}_\eps := \int_\Omega \eps(x)E\cdot F\,dx. 
 $$  
 The crucial observation is that $\mathcal L$ is a selfadjoint operator in the Hilbert
 space $(L^2(\Omega;\R^3),\skp{\cdot}{\cdot}_\eps)$. In the following we  write $E\perp_\eps
 F$ if $\skp{E}{F}_\eps=0$ and analogously for subspaces of $L^2(\Omega;\R^3)$.  
 
\begin{prop} \label{prop:HelmholtzDecompositionN}  
  Assume (A1),(A2) and $1<p<\infty$. Then we have 
  $$
    L^{p'}(\Omega;\R^3)=X^{p'}\oplus Y^{p'}
    \qquad\text{with}\quad  (Y^p)^{\perp_\eps} = X^{p'}, Y^p = (X^{p'})^{\perp_\eps}.
  $$
\end{prop}
\begin{proof}
    First of all, $X^{p'}$, $Y^{p'}$ are closed subspaces of $L^{p'}(\Omega;\R^3)$ and the intersection of
    these subspaces is $\{0\}$.  Indeed, for $E\in X^{p'}\cap Y^{p'}$ we have $E=\nabla u$ 
    for some $u\in W^{1,p'}(\Omega)$ with 
    $$
      \int_\Omega \eps(x)\nabla u\cdot\nabla\phi\,dx       
      = 0\quad\text{for all }\phi\in W^{1,p}(\Omega).
    $$
    Proposition~\ref{prop:AuschQafI} for $f=0$ then implies $\nabla u=0$, i.e., $E=0$. We thus conclude
    that the sum is direct and $X^{p'}\oplus Y^{p'}\subset L^{p'}(\Omega;\R^3)$. To prove equality we define,
    for any given $f\in L^{p'}(\Omega;\R^3)$, $\Pi f:=f+\nabla u$  where $u$ is the
    solution from Proposition~\ref{prop:AuschQafI}. Then $\Pi$ is a bounded linear operator on
    $L^{p'}(\Omega;\R^3)$ with 
    $$
      \int_\Omega \eps(x)(\Pi f)\cdot\nabla\phi\,dx
      = \int_\Omega \eps(x)(f+\nabla u)\cdot \nabla\phi\,dx  
      =  0 \qquad\text{for all }\phi\in W^{1,p}(\Omega)
    $$
    in view of~\eqref{eq:BVPAuschQafI}.
    We conclude $\Pi:L^{p'}(\Omega;\R^3)\to X^{p'}$, so $\nabla u\in Y^{p'}$ implies 
    $f = \Pi f  - \nabla u \in X^{p'}\oplus Y^{p'}$.  We thus obtain
    $$
      L^{p'}(\Omega;\R^3)=X^{p'}\oplus Y^{p'}.
      %\quad\text{with bounded projectors }
      %\Pi :L^{p'}(\Omega;\R^3)\to X^{p'},\;
      %\id-\Pi:L^{p'}(\Omega;\R^3)\to Y^{p'}.
    $$
    The equality $(Y^p)^{\perp_\eps} = X^{p'}$ holds by definition. To show $Y^p = (X^{p'})^{\perp_\eps}$ first note
    that the inclusion $\subset$ is trivial. So let us assume $g\in (X^{p'})^{\perp_\eps}$, in particular $g\in
    L^p(\Omega;\R^3)$.  Since $f:=\eps(x)^{-1}(\nabla\times\Phi)\in X^{p'}$ whenever $\Phi\in
    C_0^\infty(\Omega;\R^3)$, we get
    \begin{equation} \label{eq:orthogonalityN}
      \int_\Omega (\nabla\times\Phi)\cdot g\,dx 
      = \int_\Omega \eps(x)f\cdot g(x)\,dx
      = 0\qquad\text{for all }\Phi\in C_0^\infty(\Omega;\R^3).
    \end{equation}
    This and Lemma~\ref{lem:curlfree} imply that $g$ must be a gradient, which proves the claim. 
  \end{proof}

  \begin{rem}
    In the case of homogeneous and isotropic permittivities, i.e., $\eps(x)\equiv \eps_0$ with $\eps_0\in
    (0,\infty)$, the  above Helmholtz Decomposition Theorem was proved by 
    Fujiwara-Morimoto in~\cite{FujiwaraMorimoto} for smooth domains and Simader and 
    Sohr~\cite[Theorem~1.4]{SimaderSohr} extended this result to $C^1$-domains. We refer to
    \cite[Theorem~11.1]{FabMedMit_BoundaryLayers},\cite[Theorem~1.3]{GengShen} for alternative proofs and
    variants of this result in Lipschitz domains.
  \end{rem}
  
  In order to use the Fredholm theory of symmetric compact operators we first show the
  following.  
 
 \begin{prop}
   Assume (A1),(A2). Then $\mathcal V$ is a closed subspace of $H^1(\Omega;\R^3)$.
 \end{prop}
 \begin{proof}
   The space $\mathcal V$ is a subset of $F(\Omega,\eps,\nu)$ defined in \cite{FilPro} given that all vector
   fields $E\in\mathcal V$ satisfy $$
     E\in L^2(\Omega;\R^3),\quad \nabla\times E\in L^2(\Omega;\R^3),\quad 
     \nabla\cdot (\eps E)=0 \text{ in }\Omega,\quad 
     (\eps E)\cdot\nu=0 \text{ on }\partial\Omega.
   $$
   By \cite[Theorem~1.1]{FilPro} and (A1),(A2) the norms
   $\|\cdot\|_{H^1(\Omega;\R^3)}$ and $\|\cdot\|$ from~\eqref{eq:innerproduct} are equivalent on
   $F(\Omega,\eps,\nu)$ and in particular on $\mathcal V$, which proves the result.
 \end{proof}

%   for standard Sobolev spaces. To see this note that assumption (A2) and the Friedrichs-Gaffney
%  inequality for $E\in\mathcal V$ \cite[Theorem~3.7]{CreoLancia_Gaffney} imply
% \begin{equation} \label{eq:norms}
%   \int_\Omega   (\nabla\times E)^T \mu(x)^{-1}(\nabla\times E)\,dx 
%   \sim  \int_\Omega |\nabla\times E|^2\,dx
%   \sim \int_\Omega |\nabla E|^2 \,dx. 
% \end{equation}
% Since nontrivial constant vector fields do not belong to $\mathcal V$, this norm is in fact
% equivalent to the standard $H^1(\Omega;\R^3)$-norm on that subspace. 

Note that \cite[Theorem~1.1]{FilPro} requires the exterior ball condition for $\Omega$, which is why we
 included it in (A1). As a consequence of the previous proposition, $\mathcal V$ inherits all  embeddings of
 $H^1(\Omega;\R^3)$.
 The Rellich-Kondrachov Theorem as well as the definitions of $\mathcal V,X^p$ imply 
\begin{equation}\label{eq:embeddingsN}
  \mathcal V  \hookrightarrow X^p  
  \quad\text{boundedly for }1\leq p\leq 6
  \text{ and compactly for } 1\leq p<6
\end{equation}
given that the Sobolev-critical exponent for the embeddings of $H^1(\Omega;\R^3)$ into $L^p(\Omega;\R^3)$ is
$p=6$ by our assumption on the domain regularity (A1). The Riesz Representation Theorem in $\mathcal V$
implies that $\mathcal L+1$ has a bounded resolvent from $\mathcal V'$ into $\mathcal V$, which, by
\eqref{eq:embeddingsN}, is compact as an operator from $X^{p'}$ to $X^p$ for $1\leq p<6$ and in particular
from the Hilbert space $X^2$ into itself.
%Lax-Milgram Lemma, 
We write 
$$
  \sigma(\mathcal L)
  := \big\{ \lambda\in\R: \mathcal L(E)=\lambda E \text{ for some } E\in\mathcal  V\sm\{0\}\big\}
$$
for the spectrum of $\mathcal L$ and denote by $\Eig_{\lambda}\subset \mathcal V$ the (possibly trivial)
eigenspace of $\mathcal L$ associated with $\lambda\in\R$. In view of \eqref{eq:embeddingsN} we have 
$\Eig_{\lambda}\hookrightarrow X^p$ and, for notational convenience, we write
$\Eig_{\lambda}^p$ whenever this set is considered as a closed subspace of $X^p$ and hence of
$L^p(\Omega;\R^3)$. Fredholm theory for selfadjoint operators in Hilbert
spaces gives the following.

\begin{prop}\label{prop:linear_theoryN}
  Assume (A1),(A2), $1\leq p\leq 6$ and $\lambda\in\R$.
  \begin{itemize}
    \item[(i)] The selfadjoint operator $\mathcal L$ in the Hilbert space $(X^2,\skp{\cdot}{\cdot}_\eps)$ has
    countably many eigenvalues such that the corresponding eigenfunctions form an orthonormal basis of
    $(X^2,\skp{\cdot}{\cdot}_\eps)$ and of $(\mathcal V,\skp{\cdot}{\cdot})$. All of these eigenvalues are
    positive and they tend to $+\infty$.
    \item[(ii)] The linear problem \eqref{eq:LinBVPN} with $g\in X^{p'}$ has a weak solution 
    if and only if $g \perp_\eps \Eig_{\lambda}^p$. The solution is unique up to elements of
    $\Eig_{\lambda}$.
    \item[(iii)] For all $\lambda\in \R\sm \sigma(\mathcal L)$ the resolvent 
    $(\mathcal L-\lambda)^{-1}:X^{p'}\to \mathcal V$ is bounded for all $p\leq 6$ and
    it is compact provided that $p<6$. The analogous statement is true for 
    $\lambda\in \sigma(\mathcal L)$ if $X^{p'},\mathcal V$ are replaced by the subspaces 
    $
      \{g\in X^{p'}: g\perp_\eps \Eig_\lambda^p\}$ and 
      $\{E\in \mathcal V: E\perp  \Eig_\lambda \}$.
  \end{itemize}
\end{prop} 
\begin{proof}
  The claims are standard except for the positivity and the unboundedness of the eigenvalues. 
  The eigenvalues are unbounded from above given that the corresponding Rayleigh quotients  are unbounded
  from above over $\mathcal V$. To see this one may choose nontrivial divergence-free test functions
  with shrinking support, e.g., $E_\tau(x) := \chi(\tau|x-y|^2) (y_2-x_2,x_1-x_1,0)$ for suitable $\chi\in
  C_0^\infty(\R), y\in\Omega$ and $\tau\nearrow \infty$. Moreover, every eigenpair $(E,\lambda)\in
  \mathcal V\times\R$ of $\mathcal L$ satisfies, in view of (A2),
  $$ 
    \lambda \|E\|_{\eps}^2 
       = \lambda\int_\Omega \eps(x)E\cdot E\,dx 
       = \int_\Omega \mu(x)^{-1}(\nabla\times E)\cdot (\nabla\times E)\,dx
       \gtrsim   \int_\Omega |\nabla\times E|^2\,dx. 
   $$
  As a consequence, eigenvalues $\lambda$ satisfy $\lambda\geq 0$ and $\lambda=0$ if and
  only if the associated eigenfunction satisfies $\nabla \times E=0$. Given that $E\in\mathcal V$, the latter
  is equivalent to $E=0$ because Lemma~\ref{lem:curlfree} implies $E=\nabla u$ for some $u\in H^1(\Omega)$ and thus, by
  definition of $\mathcal V$, $\|\nabla u\|_\eps^2 = \skp{E}{\nabla u}_\eps=0$.
  As a consequence, $\lambda$ must be positive. 
\end{proof}
% \begin{proof}
%   The functional $\mathcal I$ is well-defined in view of the embedding \eqref{eq:embeddingsN} and even
%   strictly convex, coercive and continuous on $\mathcal V$ because of~\eqref{eq:norms} and $\lambda\geq 0$.
%   The Direct Method of the Calculus of Variations provides a unique minimizer and hence a unique critical
%   point of $\mathcal I$ over $\mathcal V$. This critical point is the unique weak solution
%   of~\eqref{eq:LinBVPN}. The compactness for $p<6$ is a consequence of the boundedness of $\mathcal
%   L^{-1}:X^{p'}\to\mathcal V$ combined with the
%   embedding~\eqref{eq:embeddingsN}.  
% \end{proof}

% 
%  Proposition~\ref{prop:linear_theoryN} of course provides an solution theory for
%  \eqref{eq:LinBVPN} with $E\in \mathcal V$, $g\in X^{p'}$ and $\omega^2\in [0,\infty) \sm
%   \sigma(\mathcal L)$. A change of coordinates $E\mapsto \eps(x)^{1/2}E$ leads to the equality
%   $\sigma(\mathcal L)=\sigma(\tilde{\mathcal L})$.
  
  Now we extend the considerations to the linear problem~\eqref{eq:LinBVPN} for right hand sides
  $g\in X^{p'}\oplus Y^q$ for suitable $q\in [1,\infty]$. The right ansatz for the solution space is  
  $\mathcal V \oplus Y^q$.  A weak solution $E\in \mathcal V\oplus Y^q$ of \eqref{eq:LinBVPN}, 
  with $E=E_1+E_2,E_1\in\mathcal V,E_2\in Y^q$, is supposed to satisfy
  \begin{equation*}
    \int_\Omega   \mu(x)^{-1}(\nabla\times E_1)\cdot(\nabla\times \Phi_1)\,dx
    -  \lambda\int_\Omega  \eps(x)E\cdot\Phi\,dx
    = \int_\Omega \eps(x)g\cdot\Phi\,dx 
  \qquad\text{for all }\Phi\in\mathcal V\oplus Y^q.
  \end{equation*}
  This extends the notion of a weak solution in $\mathcal V$ given by~\eqref{eq:weaksolutionI}.
  To have well-defined integrals we assume\footnote{Defining the integral $$
    \int_\Omega \eps(x)g\cdot \Phi\,dx
    :=  \int_\Omega \eps(x)g_1\cdot \Phi_1\,dx + \int_\Omega \eps(x)g_2\cdot \Phi_2\,dx
  $$
  by formal orthogonality with respect to $\skp{\cdot}{\cdot}_\eps$ we can even relax the assumption $q\geq
  \max\{2,p'\}$ to $q\geq 2$. Indeed, $g_1\in X^{p'},\Phi_1\in\mathcal V\subset X^p, 1\leq p\leq 6$ justifies
  the first integral and $g_2,\Phi_2\in Y^q, q\geq 2$ justifies the second one.  
  However, given that our applications only concern the range $2<p<6$, this extension is not needed in this
  paper.}  
  $1\leq p\leq 6$ as before as well as $q\geq \max\{2,p'\}$. 
   By orthogonality, this is equivalent to
  \begin{align*}
    \int_\Omega   \mu(x)^{-1}(\nabla\times E_1)\cdot(\nabla\times \Phi_1)\,dx
    -  \lambda\int_\Omega  \eps(x)E_1\cdot\Phi_1\,dx
    &= \int_\Omega \eps(x)g_1\cdot\Phi_1\,dx 
  \qquad\text{for all }\Phi_1\in\mathcal V, \\
    -  \lambda\int_\Omega  \eps(x)E_2\cdot\Phi_2\,dx
    &= \int_\Omega \eps(x)g_2\cdot\Phi_2\,dx
  \qquad\text{for all }\Phi_2\in Y^q.
  \end{align*}
  From Proposition~\ref{prop:linear_theoryN} we get the solution theory for the first equation and the
  solution of the second equation is trivial by our choice of the solution space. We obtain  the
  following.
  
  \begin{thm}\label{thm:linear_theoryN}
    Assume (A1),(A2) as well as $g=g_1+g_2$ where $g_1\in X^{p'}, g_2\in Y^q$ for exponents 
    $1\leq p\leq 6, \max\{2,p'\}\leq q\leq \infty$ and $\lambda\in\R$.     
    \begin{itemize}
      \item[(i)] If $\lambda\in\R\sm (\sigma(\mathcal L)\cup\{0\})$ then the unique weak solution
      $E\in \mathcal V\oplus Y^q$ of~\eqref{eq:LinBVPN} is given by
      $$
        E= (\mathcal L-\lambda)^{-1}g_1 + \lambda^{-1}g_2.
      $$
      \item[(ii)] If $\lambda\in \sigma(\mathcal L)$ then \eqref{eq:LinBVPN} admits weak solutions if and only
      if $g_1\perp_\eps \Eig^p_{\lambda}$. In this case all weak solutions   
      $E\in \mathcal V\oplus Y^q$ are given by
      $$
        E\in (\mathcal L-\lambda)^{-1}g_1 + \lambda^{-1}g_2 + \Eig_\lambda.
      $$
      \item[(iii)]  If $\lambda=0$ then \eqref{eq:LinBVPN} admits weak solutions if and
      only if $g_2=0$. In this case all weak solutions   
      $E\in \mathcal V\oplus Y^q$ are given by
      $$
        E \in \mathcal L^{-1}g_1  + Y^q.
      $$
    \end{itemize}
  \end{thm}
  
  Later, in the discussion of the nonlinear Neumann problem for $\lambda=\omega^2\in\sigma(\mathcal L)$, it
  will turn out convenient to choose $g_1$ as in (ii). To this end, we introduce for $1\leq p\leq 6$ and
  $\lambda\in\R$
  \begin{align*}% \label{eq:Xpmomega2}    
    X^{p'}_{\lambda}
    :=\Big\{ f \in X^{p'}: \int_\Omega \eps(x)f\cdot\phi\,dx = 0 \text{ for all 
    }\phi\in \Eig_{\lambda}^p \Big\}.  
  \end{align*} 
  This definition makes sense for  $1\leq p\leq 6$ and $\lambda\in\R$. We then have 
  $g_1\in X^{p'}, g_1\perp_\eps \Eig^p_{\lambda}$ if and only if $g_1\in X^{p'}_{\lambda}$. 
  This is why the function space $X^{p'}_{\lambda}$ and its properties will be needed later on.   
   
  \begin{prop}\label{prop:Xlambda}
    Assume $\frac{6}{5}\leq p\leq 6$ and $\lambda\in\R$. Then $X^{p'}  = X^{p'}_{\lambda} \oplus
    \Eig_{\lambda}^{p'}$ with 
    $$
      (X^{p'}_{\lambda})^{\perp_\eps} =\Eig_{\lambda}^p\oplus Y^p, \qquad
      (\Eig_{\lambda}^{p'}\oplus Y^{p'})^{\perp_\eps} =   X^p_{\lambda}.
    $$  
  \end{prop}
  \begin{proof}
    We only prove $(X^{p'}_{\lambda})^{\perp_\eps} \subset \Eig_{\lambda}^p\oplus Y^p$, so let $f\in
    (X^{p'}_{\lambda})^{\perp_\eps} \subset L^p(\Omega;\R^3)$ and define  
    $$
      \Pi(f) := f - \sum_{i=1}^n  \skp{f}{\phi_i}_\eps \phi_i \in L^p(\Omega;\R^3)
      \qquad\text{where }\{\phi_1,\ldots,\phi_n\} \text{ is an ONB of
      }(\Eig_{\lambda}^p,\skp{\cdot}{\cdot}_\eps).  
    $$
%     This is well-defined due to $\frac{6}{5}\leq p\leq 6$. By definition of $f$ we have $\skp{f}{g}_\eps=0$ 
%     for all $g\in X^{p'}_{\lambda}$. This even implies $\skp{\Pi(f)}{g}_\eps = 0$ for all $g\in
%     X^{p'}_{\lambda}$ by definition of $X^{p'}_\lambda$.  But then the ONB property implies that 
%     the same orthogonality relation also holds for eigenfunctions $g\in \Eig_{\lambda}^{p'}$. 
%     So $X^{p'}_{\lambda}\oplus \Eig_{\lambda}^{p'}= X^{p'}$ implies 
%     $$
%       \skp{\Pi(f)}{g}_\eps  = 0 \quad\text{for all } g\in X^{p'}.
%     $$
 	With this definition we find $\Pi(f) \in (X^{p'})^{\perp_\eps} =  Y^p$ by
 	Proposition~\ref{prop:HelmholtzDecompositionN}. This implies $f\in \Eig_{\lambda}^p\oplus Y^p$.
  \end{proof}

  \section{Proof of Theorem~\ref{thm:N}} \label{sec:TheoremN}
  
  The nonlinear Neumann problem~\eqref{eq:NLCurlCurlN} reads  
  $$ 
    (\mathcal L-\omega^2)E = P \qquad\text{where } P:=\eps(x)^{-1}f(x,E) \text{ and }E\in\mathcal
    V\oplus\mathcal W. 
  $$
  Aiming for a dual formulation of this problem, we solve the linear problem and treat the resulting equation
  as a problem for the vector field $P$.  The inversion of  the map $E\mapsto P$ is possible thanks to our
  assumption on the nonlinearity $f$. In the Appendix (Proposition~\ref{prop:Apsi}) we show that (A3)
  implies that an inverse $\psi(x,\cdot):=f(x,\cdot)^{-1}$ exists almost everywhere with the
  following properties:
   \begin{itemize}
  \item[(A3')]  $\psi:\Omega\times\R^3\to\R^3$ is measurable with $\psi(x,P)=\psi_0(x,|P|)|P|^{-1}P$ where,
  for almost all $x\in\Omega$,
  \begin{align*}
    z&\mapsto \psi_0(x,z)  \text{ is positive, increasing and differentiable on }(0,\infty), \\
    z&\mapsto z^{-1}\psi_0(x,z) \text{ is decreasing on }(0,\infty).
  \end{align*}
  Moreover, there are $c_1,c_2>0$ and $2<p<6$ such that for almost all $x\in\Omega$ and $z>0$ we have
  \begin{align} \label{eq:psi_estimate}
    \int_0^z \psi_0(x,s)\,ds -\frac{1}{2} \psi_0(x,z)z
    \geq c_1|z|^{p'}
    \geq c_2 \psi_0(x,z) z
  \end{align}
  %where $\Gamma(x)\sim (1+|x|)^{-\alpha}$ for $0<\alpha<3$ and $2<p<6, \max\{2,6-2\alpha\}<q<\infty$.
\end{itemize}
 We anticipate Proposition~\ref{prop:Apsi} for the sake of the presentation.
  
  \begin{prop}\label{prop:psi}
    Assume that $f:\Omega\times\R^3\to\R^3$ satisfies~(A3). Then
    $\psi(x,\cdot):=f(x,\cdot)^{-1}$ exists for almost all $x\in\Omega$ and satisfies (A3').
  \end{prop}
  
  As a consequence, solving the nonlinear Neumann problem amounts to solving
  the quasilinear problem 
  \begin{equation} \label{eq:quasilinear}
    (\mathcal L-\omega^2)\big(\psi(x,\eps(x)P)\big)  = P.
  \end{equation}
  The case distinction (i),(ii),(iii) in Theorem~\ref{thm:linear_theoryN} leads to a separate discussion of
  \eqref{eq:quasilinear} according to the following cases:
  $$
    \text{(I)}\quad \omega^2 \in (0,\infty)\sm\sigma(\mathcal L)\qquad\text{or}\qquad    
    \text{(II)}\quad \omega^2 \in\sigma(\mathcal L) \qquad\text{or}\qquad 
    \text{(III)}\quad \omega^2 = 0.
  $$
  Case (I) will be treated in full detail whereas our presentation of the cases (II),(III) focusses on the
  modicifications with respect to (I).
  
  \subsection{\textbf{Case (I)}} \; Here  we assume $\omega^2\in (0,\infty)\sm \sigma(\mathcal L)$.
  The first step is to prove that the original problem  is equivalent to finding ground and bound states 
  of the functional
  \begin{align}
    J(P):=  \int_{\Omega}
    \Psi(x,\eps(x)P) \,dx + \frac{1}{2\omega^2} \int_\Omega P_2\cdot \eps(x) P_2\,dx
    - \frac{1}{2}\int_{\Omega} (\mathcal L-\omega^2)^{-1}P_1\cdot \eps(x)P_1\,dx
    \label{eq:functionalJdualN} 
  \end{align}
  for $P\in Z:= X^{p'}\oplus Y^2$. It is straightforward to check $J\in C^1(Z)$ with Fr\'{e}chet derivative  
  \begin{align*}
      J'(P)[h] = \int_{\Omega}\psi(x,\eps(x)P)\cdot  \eps(x)h\,dx
      +  \omega^{-2} \int_\Omega P_2\cdot \eps(x)h_2\,dx
      - \int_{\Omega} (\mathcal L-\omega^2)^{-1}P_1\cdot \eps(x)h_1\,dx
  \end{align*}
  for all $h\in Z$. Here one uses that $(\mathcal L-\omega^2)^{-1}$ is symmetric with respect to
  $\skp{\cdot}{\cdot}_\eps$. Exploiting the formulas for $(X^{p'})^{\perp_\eps}$ and 
  $(Y^2)^{\perp_\eps}$ from Proposition~\ref{prop:HelmholtzDecompositionN} we find that the Euler-Lagrange
  equation of $J$ reads
  \begin{align}
    \psi(x,\eps(x)P) = (\mathcal L-\omega^2)^{-1}P_1 -\omega^{-2}P_2 
     \qquad\text{for } P_1\in X^{p'},\;P_2\in Y^2. \label{eq:dualequationN} 
  \end{align}
   
\begin{lem} \label{lem:equivalenceI}
  Assume (A1),(A2),(A3) and $\omega^2\in (0,\infty)\sm \sigma(\mathcal L)$. Then  $I'(E)=0,E\in \mathcal
  V\oplus \mathcal W$\; if and only if\; $J'(P)=0,P\in Z$,
  where $P,E$ are related to each other via $P= \eps(x)^{-1}f(x,E)$, $E=\psi(x,\eps(x)P)$
  with $\psi$ as in Proposition~\ref{prop:psi}.
\end{lem}
\begin{proof}
  Assume $I'(E)=0$ where $E=E_1+E_2$ for $E_1\in \mathcal V, E_2\in \mathcal W$, define 
  $P:=\eps(x)^{-1}f(x,E)$. From $E\in \mathcal V\oplus \mathcal W\subset L^p(\Omega;\R^3)$ and the growth
  properties of $f$ from (A3) we infer $P\in L^{p'}(\Omega;\R^3)$. 
  We write $P=P_1+P_2$ according to the Helmholtz
  Decomposition from Proposition~\ref{prop:HelmholtzDecompositionN}. According to the formula of $I$ 
  from \eqref{eq:defI} we get for $\Phi_1\in\mathcal V,\Phi_2=0$  
  \begin{align*}
    0 &= \int_{\Omega} \mu(x)^{-1}(\nabla\times E_1)\cdot (\nabla\times \Phi_1)\,dx 
    - \omega^2 \int_\Omega \eps(x)E\cdot \Phi_1\,dx
%    \b{-\frac{\omega^2}{2}\int_\Omega (E_1^T\eps(x)E_1+E_3^T\eps(x)E_2)\,dx}
    -  \int_{\Omega} f(x,E)\cdot \Phi_1\,dx \\  
      &= \int_{\Omega} \mu(x)^{-1}(\nabla\times E_1)\cdot (\nabla\times \Phi_1)\,dx 
    - \omega^2 \int_\Omega \eps(x)E_1\cdot \Phi_1\,dx
%    \b{-\frac{\omega^2}{2}\int_\Omega (E_1^T\eps(x)E_1+E_3^T\eps(x)E_2)\,dx}
    -  \int_{\Omega} \eps(x)P_1\cdot \Phi_1\,dx. 
  \end{align*}
  Choosing instead $\Phi_1=0\Phi_2\in\mathcal W$ we get
  \begin{align*}
    0 &=  - \omega^2 \int_\Omega \eps(x)E_2\cdot \Phi_2\,dx
%    \b{-\frac{\omega^2}{2}\int_\Omega (E_1^T\eps(x)E_1+E_3^T\eps(x)E_2)\,dx}
    -  \int_{\Omega} \eps(x)P_2\cdot \Phi_2\,dx    
  \end{align*}
  for all $\Phi_2\in \mathcal W$. This implies   
  $$
    E_1 = (\mathcal L-\omega^2)^{-1}P_1,\qquad 
    E_2 = -\omega^{-2}P_2.
  $$
  In particular, $P_2 = -\omega^2 E_2\in \mathcal W = Y^p \subset Y^2$, so 
  $$
    P=P_1+P_2\in X^{p'}\oplus Y^2=Z
    \quad\text{and } \psi(x,\eps(x)P)=E=E_1+E_2 = (\mathcal L-\omega^2)^{-1}P_1 -\omega^{-2}P_2.
  $$ 
  So \eqref{eq:dualequationN} holds and we conclude  $J'(P)=0$.
 To prove the reverse implication assume $J'(P)=0$ for $P\in X^{p'}\oplus Y^2$ so that
 \eqref{eq:dualequationN} holds.
 This implies $P\in L^{p'}(\Omega;\R^3)$ and hence $E:=\psi(x,\eps(x)P)\in L^p(\Omega;\R^3)$ in view of (A3').
 Recall that (A3') is equivalent to (A3) thanks to Proposition \ref{prop:psi}.
 From \eqref{eq:dualequationN} we even get $E_1= (\mathcal L-\omega^2)^{-1}P_1\in \mathcal V$ by
 Proposition~\ref{prop:linear_theoryN}(iii), so $E\in \mathcal V\oplus \mathcal W$. 
 So~\eqref{eq:dualequationN} and Theorem~\ref{thm:linear_theoryN}(i) imply $I'(E)=0$. 
\end{proof}

Knowing that the dual problem and the original one are equivalent, we may now focus on proving the
existence of critical points of $J$ with the aid of Theorem~\ref{thm:CPTheorem}. 
 
  \begin{prop} \label{prop:PScondition}
    Assume (A1),(A2),(A3'). Then $J$ from \eqref{eq:functionalJdualN} satisfies the Palais-Smale condition.
  \end{prop}
  \begin{proof}
    Let $(P_n)$ be a Palais-Smale sequence for $J$,
    so $J(P_n)\to c\in\R$ and $J'(P_n)\to 0$. From (A3') we get
    $$
      c+o(1)\|P_n\|_{p'}
      = J(P_n)-\frac{1}{2}J'(P_n)[P_n]
      = \int_\Omega \Psi(x,P_n) - \frac{1}{2}\psi(x,P_n)\cdot P_n\,dx
      \gtrsim \int_\Omega |P_n|^{p'}\,dx
      = \|P_n\|_{p'}^{p'},
    $$
    so $(P_n)$ is bounded in $L^{p'}(\Omega;\R^3)$. This and $J(P_n)\to c$ imply that $(P_n^2)$ is
    bounded in $L^2(\Omega;\R^3)$ as well, so  we may assume $P_n^1\wto P^1$ in $X^{p'}$ and 
    $P_n^2\wto P^2$ in $Y^2$. From $J'(P_n)\to 0, P_n\wto P$ and the compactness of $(\mathcal
    L-\omega^2)^{-1}:X^{p'}\to X^p$ (Proposition~\ref{prop:linear_theoryN}(iii)) we get
    as $n\to\infty$.
   \begin{align*}
      o(1) 
      &= (J'(P_n)-J'(P))[P_n-P] \\
      &= \int_\Omega \big(\psi(x,\eps(x)P_n)-\psi(x,\eps(x)P)\big)\cdot \eps(x)(P_n-P) +
      \omega^{-2}(P_n^2-P^2)\cdot \eps(x)(P_n^2-P^2)\,dx \\
      &- \int_\Omega (\mathcal L-\omega^2)^{-1}(P_n^1-P^1)\cdot \eps(x)(P_n^1-P^1)\,dx \\
      &= \int_\Omega \big(\psi(x,\eps(x)P_n)-\psi(x,\eps(x)P)\big)\cdot \eps(x)(P_n-P) +
      \omega^{-2}(P_n^2-P^2)\cdot \eps(x)(P_n^2-P^2)\,dx + o(1). 
   \end{align*}  
    The integrand is nonnegative and hence converges to zero in $L^1(\Omega)$. By a Corollary of the
    Riesz-Fischer Theorem, a subsequence, still denoted by $(P_n)$, is pointwise almost everywhere bounded by
    some function $H\in L^1(\Omega)$ and converges to zero pointwise almost everywhere. Since $\psi(x,\cdot)$
    is strictly monotone for almost all $x\in\Omega$ and $\eps(x)$ is uniformly positive definite, we deduce
    $P_n\to P$ and $P_n^2\to P^2$ pointwise almost everywhere.
    Furthermore,  combining
    $$ 
      \big(\psi(x,\eps(x)P_n)-\psi(x,\eps(x)P)\big)\cdot \eps(x)(P_n-P) + \omega^{-2}(P_n^2-P^2)\cdot
      \eps(x)(P_n^2-P^2) \leq H \qquad\text{on }\Omega
    $$
    with the estimates $\psi(x,\eps(x)P_n)\cdot \eps(x)P_n\gtrsim |P_n|^{p'}$ and $|\psi(x,\eps(x) P_n)|\les
    |P_n|^{p'-1}$ from (A3') gives
    $$
      |P_n|^{p'} + |P_n^2|^2 \leq \hat H  \qquad\text{on }\Omega
    $$
    where $\hat H\in L^1(B)$ is defined in terms of $H,P,\psi(x,\eps(x)P)$. So the Dominated Convergence
    Theorem implies  $P_n\to P$ in $L^{p'}(\Omega)$ and $P_n^2\to P^2$ in $L^2(\Omega)$, hence $\|P_n-P\|\to 0$. In
    particular, $J'(P)=0$ and the claim is proved.
 \end{proof}   
 
 We finally prove the existence of critical points. As usual, a ground state is a nontrivial critical point
 with least energy among all nontrivial critical points. We will see that ground states for $J$   give
 rise to ground states for $I$ and hence minimal energy solutions for the original problem.
    
 \begin{thm} \label{thm:dual}
   Assume (A1),(A2),(A3') and $\omega^2\in (0,\infty)\sm \sigma(\mathcal L)$. Then $J$ admits a ground state
   and infinitely many bound states.
 \end{thm}
 \begin{proof}
  The Banach space $Z=X^{p'}\oplus Y^2$ satisfies $(G_1)$ with $Z^+=Z,Z^-=\{0\}$. We check that the functional
  $J\in C^1(Z)$ has the Mountain Pass Geometry $(G_2),(G_3)$. Indeed, $(G_2)$ is straightforward and as an 
  $m$-dimensional subspace $Z_m$ with the properties required in $(G_3)$ we may
  choose  the span of $m$ linearly independent eigenfunctions associated with eigenvalues
  $\lambda_1,\ldots,\lambda_m \in (\omega^2,\infty)\cap \sigma(\mathcal L)$. With this choice, we get for a
  linear combination of eigenfunctions $P:= \sum_{i=1}^m c_i \phi_i \in Z_m$
  \begin{align*}
    \int_{\Omega} (\mathcal L-\omega^2)^{-1}P\cdot \eps(x)P\,dx
    &= \sum_{i,j=1}^m c_ic_j \int_{\Omega} (\mathcal L-\omega^2)^{-1}\phi_i\cdot \eps(x)\phi_j\,dx \\
    &= \sum_{i,j=1}^m \frac{c_ic_j}{\lambda_i-\omega^2}  \int_{\Omega} \phi_i\cdot
    \eps(x)\phi_j\,dx  \\
    &= \sum_{i=1}^m \frac{c_i^2}{\lambda_i-\omega^2}  \|\phi_i\|_\eps^2 \\
    &\geq \min\{ (\lambda_i-\omega^2)^{-1}:i=1,\ldots,m\} \cdot \|P\|_\eps^2 \\
    &\geq c_m \|P\|^2
  \end{align*}
  for some $c_m>0$. Hence, $P=P_1$ and $P_2=0$ gives for some $C>0$ 
  $$
    J(P)=  \int_{\Omega} \Psi(x,\eps(x)P) \,dx 
    - \frac{1}{2}\int_{\Omega} (\mathcal L-\omega^2)^{-1}P\cdot \eps(x)P\,dx
     \leq C \|P\|_{p'}^{p'} - \frac{c_m}{2} \|P\|^2,
  $$
  which implies $J(P)\to -\infty$ as $P\in Z_m,\|P\|\to\infty$.
  Finally, Proposition~\ref{prop:PScondition} shows that $J$ satisfies the Palais-Smale condition, so $(G_4)$
  holds by Remark~\ref{rem:Abstract}(a). Theorem~\ref{thm:CPTheorem} then implies the existence of infinitely
  many finite energy solutions. Finally, a critical point exists  at the mountain pass level
  \cite[Theorem~1.15]{Willem}  and this solution is in fact a ground state by~\cite[Theorem~4.2]{Willem}.
\end{proof}

\textbf{\textit{Proof of Theorem~\ref{thm:N} for $\omega^2\in (0,\infty)\sm \sigma(\mathcal L)$:}}\; 
Since $f$ satisfies~(A3), the function $\psi(x,\cdot):= f(x,\cdot)^{-1}$
satisfies~(A3') by Proposition~\ref{prop:psi}. So Theorem~\ref{thm:dual}  implies that the functional $J$
has a ground state $P^\star\in Z\sm\{0\}$. By Lemma~\ref{lem:equivalenceI},
$E^\star(x):=\psi(x,P^\star(x))$ satisfies $E^\star\in \mathcal V\oplus\mathcal W\sm\{0\}$ as well as
$I'(E^\star)=0$. In fact, $E^\star$ is even a ground state, which is proved\footnote{ 
One may argue as in the proof of \cite[Theorem~15]{Man_Nonlocal}. In the notation of that paper,
\begin{align*}
  X&:= \mathcal V\oplus \mathcal W,\qquad 
  Y:= Z = X^{p'}\oplus Y^2,\qquad  G:=J,\qquad \varphi(h) := \int_\Omega h\,dx \\
  Q_1(E,\tilde E) &:= \int_\Omega \mu(x)^{-1}(\nabla\times E)\cdot (\nabla\times \tilde
  E)-\omega^2\eps(x)E\cdot \tilde E\,dx,\qquad 
  Q_2(E,\tilde E) := \int_\Omega (\mathcal L-\omega^2)^{-1}E\cdot \tilde E\,dx.
\end{align*} 
The theorem requires the equivalence of the original and the dual problem, see eq.
(10) in \cite{Man_Nonlocal}, which we checked in Lemma~\ref{lem:equivalenceI}. 
Even though $F,G$ are in general not twice continuous differentiable, the same argument as in
\cite{Man_Nonlocal} gives the claim.} as in \cite{Man_Nonlocal}. 
Theorem~\ref{thm:dual} even provides infinitely many other nontrivial critical points of $J$ and hence of $I$,
which finishes the proof of Theorem~\ref{thm:N}.
\qed

 \subsection{\textbf{Case (II)}}
  
  We proceed as in the previous section and start by defining the energy functional  
 \begin{align}
    J(P):= \int_{\Omega}
    \Psi(x,\eps(x)P) \,dx + \frac{1}{2\omega^2} \int_\Omega P_2\cdot \eps(x)P_2\,dx
    - \frac{1}{2}\int_{\Omega} (\mathcal L-\omega^2)^{-1}P_1\cdot \eps(x) P_1\,dx,
    \label{eq:functionalJdualNomega2} 
 \end{align}
 on the Banach space $Z_{\omega^2} := X_{\omega^2}^{p'}\oplus Y^2$.  Note that by our choice of
 $Z_{\omega^2}$ the resolvent is well-defined as a linear operator acting on $P_1\in X_{\omega^2}^{p'}$. We
 have  $J\in C^1(Z_{\omega^2})$ with
    \begin{align*}
      J'(P)[h] = \int_{\Omega}\Psi(x,\eps(x)P)\cdot  \eps(x)h\,dx
      + \omega^{-2} \int_\Omega P_2\cdot \eps(x)h_2\,dx
      - \int_{\Omega} (\mathcal L-\omega^2)^{-1}P_1\cdot \eps(x) h_1\,dx
    \end{align*}
  for all $h\in Z_{\omega^2}$ and the Euler-Lagrange equation reads, in view of
  Proposition~\ref{prop:Xlambda},
  \begin{align}\label{eq:dualequationNomega2}
    \psi(x,\eps(x)P) \in (\mathcal L-\omega^2)^{-1}P_1 -\omega^{-2}P_2 +  \Eig_{\omega^2}^p.    
  \end{align}
 
  \begin{lem} \label{lem:equivalenceomega2}
  Assume (A1),(A2),(A3) and  $\omega^2\in \sigma(\mathcal L)$. Then  
    $I'(E)=0,E\in \mathcal V\oplus \mathcal W$\; if and only if\; $J'(P)=0,P\in 
    X_{\omega^2}^{p'}\oplus Y^2$,
  where $P,E$ are related to each other via $P= \eps(x)^{-1}f(x,E)$, $E=\psi(x,\eps(x)P)$
  with $\psi$ as in Proposition~\ref{prop:psi}.
\end{lem}
 \begin{proof}
   From $P\in Z_{\omega^2}\subset L^{p'}(\Omega;\R^3)$ and the properties of $\psi$ we get 
   $E:= \psi(x,\eps(x)P)\in L^p(\Omega;\R^3)$, in particular $E_2\in Y^p=\mathcal W$.  
   Now if $J'(P)=0$, then  \eqref{eq:dualequationNomega2} implies
   $$
     E= E_1+ E_2,\qquad 
     E_1\in (\mathcal L-\omega^2)^{-1}P_1 +  \Eig_{\omega^2}^p,\quad
     E_2 = -\omega^{-2}P_2 
   $$
   From $P_1\in X^{p'}$ we infer $(\mathcal L-\omega^2)E_1 = P_1$ in the weak
   sense, in particular $E_1\in\mathcal V$. Then $I'(E)=0$ follows and the  claim is proved.
 \end{proof}

  The verification of the Palais-Smale condition and the existence proof for critical points is the same as
  above - it suffices to replace $X^{p'}$ by $X^{p'}_{\omega^2}$.

 \begin{thm} \label{thm:dualomega2}
   Assume (A1),(A2),(A3')  and $\omega^2\in  \sigma(\mathcal L)$. Then  $J$ from 
   \eqref{eq:functionalJdualNomega2} satisfies the Palais-Smale condition and admits  ground states as well as
   infinitely many bound states.
 \end{thm}

\textbf{\textit{Proof of Theorem~\ref{thm:N} for $\omega^2\in  \sigma(\mathcal L)$:}}\; Same reasoning as in
the case $\omega^2\in (0,\infty)\sm \sigma(\mathcal L)$ up to replacing Theorem~\ref{thm:dual} by
Theorem~\ref{thm:dualomega2}. \qed

\subsection{\textbf{Case (III)}} 
  
  We finally deal with  the static case $\omega^2=0$ where the boundary value problem does not involve the
  permittivity matrix $\eps$ any more, so may without loss of generality assume $\eps(x) := I_{3\times 3}$ and
  ignore (A2). The energy functional is given by $J:X^{p'}\to\R$ with 
 \begin{align*}
    J(P) ;= \int_{\Omega} \Psi(x,P) \,dx  
    - \frac{1}{2}\int_{\Omega} \mathcal L^{-1}P\cdot P\,dx.% \label{eq:functionalJdualN0} 
 \end{align*} 
 We recall $0\notin \sigma(\mathcal L)$  from Proposition~\ref{prop:linear_theoryN}.
  Once more, this functional belongs to $C^1(X^{p'})$ and the Euler-Lagrange equation reads  
  \begin{align*}
    \psi(x,P) \in \mathcal L^{-1}P +  Y^p %\label{eq:dualequationN0}  
  \end{align*}
  Proceeding as above we get the following.
   
\begin{lem} \label{lem:equivalence0}
  Assume (A1),(A3) and $\omega^2=0$. Then $I'(E)=0,E\in \mathcal V\oplus \mathcal W$ if and only if
  $J'(P)=0,P\in X^{p'}$ where $P,E$ are related to each other via $P= f(x,E)$, $E=\psi(x,P)$ and $\psi$ is
  given by Proposition~\ref{prop:psi}.
\end{lem} 
%      \begin{proof}
%      Given the mapping properties of $\mathcal L^{-1}$ from
%     Proposition~\ref{prop:linear_theoryN} the formula for $J'$ is evident. 
%     Critical points are characterized by the property
%     $$
%       \int_{\Omega} \Big(\psi(x,P)- \mathcal L^{-1}P\Big)\cdot h\,dx = 0
%       \quad\text{for all } h\in X^{p'}.
%     $$
%     By Proposition~\ref{prop:HelmholtzDecompositionN} this is equivalent to
%     $
%       \psi(x,P) -\mathcal L^{-1}P\in  (X^{p'})^{\perp_\eps} = Y^p.
%     $
%   \end{proof}

 \begin{thm} \label{thm:dual0}
   Assume (A1),(A3') and $\omega^2=0$. Then $J\in C^1(X^{p'})$ satisfies the Palais-Smale condition
   and admits ground states as well as infinitely many bound states.
 \end{thm}
 
%  In the case $\omega^2=0$ one may actually apply the standard Symmetric Mountain Pass Theorem 
%  (\cite[Corollary~2.9]{AmbRab},\cite[Theorem~6.5]{Struwe}) given that the quadratic part in $P^2$ does not
%  show up in the functional any more and $\mathcal L^{-1}$ is positive definite on $X^{p'}$. In the case
%  $\omega^2>0$ this is not possible.

\textbf{\textit{Proof of Theorem~\ref{thm:N} for $\omega^2=0$:}}\;   Same reasoning as in the
case $\omega^2\in (0,\infty)\sm \sigma(\mathcal L)$ up to replacing Theorem~\ref{thm:dual} by
Theorem~\ref{thm:dual0}. \qed

\section{Proof of Theorem~\ref{thm:R3}} \label{sec:TheoremR3}

We now use the dual variational method to prove the existence of infinitely many $L^p$-solutions to
the Nonlinear time-harmonic Maxwell's equation \eqref{eq:NLCurlCurlR3}   under the assumption
\begin{align} \label{eq:assumptions_fullspace}
  \eps(x)= \eps_0 \in (0,\infty),\qquad
  \mu(x) = \mu_0 \in (0,\infty),\qquad \omega^2>0,\qquad  
  \text{$f(\cdot,E)$ satisfies (A4).} 
\end{align} 
So  the equation to solve is
$$
  \nabla\times\nabla\times E - \lambda E = f(x,E) \quad\text{in }\R^3
  \qquad\text{where }\lambda:=\omega^2 \eps_0\mu_0.
$$ 
To this end we adapt the strategy that Evequoz and Weth~\cite{EveqWeth_Dual} used to prove the existence of
dual ground states for the Nonlinear Helmholtz Equation. In order to implement the dual variational method in the
Maxwell setting, we use the Helmholtz Decomposition on $\R^3$. We recall that
$\dot{W}^{1,p}(\R^3;\R^3)$ is a homogeneous Sobolev space, i.e., the closure of test functions with respect
to $u\mapsto \|\nabla u\|_p$. We define
\begin{align*}
 %   \mathcal V 
 %   &:= \Big\{ E \in \r{H^1(\curl;\R^3)}: \int_{\R^3} E\cdot \nabla\Phi\,dx = 0 \;\text{for all
 %   }\Phi\in C_0^\infty(\R^3)\Big\}, \\
    X^{p'} &:= \Big\{ E \in L^{p'}(\R^3;\R^3) : \int_\Omega E\cdot \nabla\Phi\,dx = 0 \;\text{for all } \Phi\in
    \dot{W}^{1,p}(\R^3;\R^3)\Big\}, \\
    Y^{p'} &:= \Big\{ \nabla u : u\in \dot{W}^{1,p'}(\R^3;\R^3)\Big\}.
  \end{align*}

\begin{prop} \label{prop:HelmholtzDecomposition}
  Assume $1<p<\infty$. Then $L^{p'}(\R^3;\R^3)=X^{p'}\oplus Y^{p'}$ with
  $(X^{p'})^{\perp} = Y^p, (Y^{p'})^{\perp} = X^p$. We have, in the distributional sense,  
  \begin{equation} \label{eq:HelmholtzDec}
    \nabla\times\nabla\times E_1 = -\Delta E_1 \quad\text{for }E_1\in X^{p'},\qquad\quad 
    \nabla\times\nabla\times E_2 =0\quad\text{for }E_2\in Y^{p'}.
  \end{equation}
 % Moreover, if $F\in Z^*$ satisfies $\int_{\R^3} F(x)\cdot
 % G(x)\,dx=0$ for all $G\in X$, then $\Pi(F)=0$.
\end{prop}
\begin{proof}
  For any given $E\in L^{p'}(\R^3;\R^3)$ we define 
  $$
    \widehat{E_1}(\xi):= \hat E(\xi) - |\xi|^{-2} (\xi\cdot\hat E(\xi)) \xi,\qquad
    \widehat{E_2}(\xi):=  |\xi|^{-2} (\xi\cdot\hat E(\xi)) \xi 
  $$
  Mikhlin's Multiplier Theorem implies that $E_1,E_2$ indeed belong to $L^{p'}(\R^3;\R^3)$. It is
  straightforward to check $(X^{p'})^{\perp} = Y^p, (Y^{p'})^{\perp} = X^p$.   
  Then \eqref{eq:HelmholtzDec} follows from the   identity
  $\nabla\times\nabla\times \Phi = -\Delta \Phi + \nabla(\nabla\cdot \Phi)$ for $\Phi\in C_0^\infty(\R^3;\R^3)$ and from the fact that the curl
  operator annihilates gradients.
\end{proof}

As a consequence, it suffices to find solutions for
$$
  (-\Delta -\lambda)E_1 - \lambda E_2 = f(x,E_1+E_2) \quad\text{in }\R^3.
$$  
To prove the existence of nontrivial solutions to this problem we study a dual problem that, 
in contrast to the situation on bounded domains studied earlier, is not equivalent to the original problem.
This is due to the fact that $-\Delta -\lambda$ is not invertible on $X^{p'}$, but admits
some sort of a right inverse $\mathcal R$ defined via the Limiting Absorption Principle for the
Helmholtz equation. We define $\mathcal R:= \Re \big[(-\Delta -\lambda-i0)^{-1}\big]$, i.e., 
  \begin{equation} \label{eq:Romega2}
     \mathcal R:X^{p'}\to X^p,\quad
     g\mapsto \lim_{\eps\to 0^+} \Re\left(  \mathcal F^{-1}\left(\frac{\hat
     g}{|\cdot|^2-\lambda-i\eps}\right)\right). 
  \end{equation}
It is known \cite[Theorem~2.1]{EveqWeth_Dual} that this is a well-defined bounded linear operator for $4\leq
p\leq 6$. Note that $\frac{2(n+1)}{n-1}=4$ corresponds to the Stein-Tomas exponent and $\frac{2n}{n-2}=6$ to
the Sobolev-critical exponent for $n=3$. Introducting $P:= f(x,E_1+E_2)$ we find that it is sufficient
to find a solution to 
$$
  E_1 = \mathcal R(P^1),\qquad E_2= - \lambda^{-1} P^2.
$$ 
This leads to the study of the energy functional
  \begin{equation}\label{eq:defJ_R3}
   J(P) = \int_{\R^3} \Psi(x,P)\,dx  + \frac{1}{2\lambda} \int_{\R^3} |P^2|^2\,dx
      - \frac{1}{2}\int_{\R^3} P^1\cdot \mathcal R(P^1)\,dx
  \end{equation}  
that is well-defined on the Banach space
$$
    Z := X^{p'}\oplus (Y^2\cap Y^{p'})
    \quad\text{with norm }\|P\| := \|P\|_{p'} + \|P^2\|_2.% \sim \|P^1\|_{p'}+\|P^2\|_{p'}+\|P^2\|_2.
  $$
Note that the assumption $P_2\in Y^2$ is not sufficient. 
The functional $J$ is continuously differentiable and the Euler-Lagrange equation reads
\begin{equation} \label{eq:ELR3}
    \psi(x,P) 
     = \mathcal R(P_1) - \lambda^{-2}P_2.
\end{equation}
It turns out that solutions $P$ of this dual problem lead to solutions of the function space $\mathcal
V\oplus\mathcal W$ where
\begin{align*}
  \mathcal V := \mathcal R(X^{p'}),\quad
  \mathcal W := Y^p\cap Y^{p'},\qquad\text{so }
  \mathcal V\oplus \mathcal W\subset L^p(\R^3;\R^3).
\end{align*}
Moreover, local elliptic regularity theory implies $\mathcal V\subset W^{2,p'}_{\loc}(\R^3;\R^3)$. 
A function $E=E_1+E_2\in \mathcal V\oplus \mathcal W$ is called  a weak solution of \eqref{eq:NLCurlCurlR3} if
\begin{equation} \label{eq:WeakSolutionR3}
  \int_{\R^3} (\nabla \times E_1)\cdot (\nabla \times \Phi_1) - \lambda E\cdot\Phi \,dx 
  = \int_{\R^3} f(x,E)\cdot \Phi\,dx
  \qquad\text{for all }\Phi\in (\mathcal V \oplus \mathcal W)\cap C_0^\infty(\Omega;\R^3)
\end{equation}
and all integrals are well-defined by the choice of our space of test functions. 
 
\begin{lem} \label{lem:equivalenceR3}
  Assume \eqref{eq:assumptions_fullspace}. Then 
  $J'(P)=0,P\in Z$ implies $I'(E)=0,E\in \mathcal V\oplus \mathcal W\subset L^p(\R^3;\R^3)$ if 
  where $E=\psi(x,P)$ and $\psi$ is given by Proposition~\ref{prop:psi}.
\end{lem} 
\begin{proof}
  Let $P\in Z$ satisfy $J'(P)=0$, so \eqref{eq:ELR3} holds. From $Z\subset L^{p'}(\R^3;\R^3)$ and the
  growth conditions of $\psi$ from (A3') we deduce $E:=\psi(x,P) \in L^p(\R^3;\R^3)$.  The Euler-Lagrange
  equation gives 
  $$
    E=E_1+E_2,\qquad
    E_1 = \mathcal R(P_1),\quad
    E_2 = -\lambda^{-1} P_2 
  $$
  and thus $E_1\in\mathcal V,E_2\in\mathcal W$ and $I'(E)=0$ in the sense of \eqref{eq:WeakSolutionR3}. 
\end{proof}
  
   Now we construct  infinitely many nontrivial critical points for $J$ with the aid of
   Theorem~\ref{thm:CPTheorem}. We shall need the following local compactness property that
   generalizes~\cite[Lemma~2.2]{EVEQ}. As a new feature, it uses Div-Curl-Lemma.
  
  \begin{prop}\label{prop:LocalCompactness}
    Assume~\eqref{eq:assumptions_fullspace}. Then for any PS-sequence $(P_n)$ of $J$ there is a subsequence
    $(P_{n_j})$ and a critical point $P$ of $J$ such that 
    $$
      P_{n_j}\wto P\quad\text{in }Z,\qquad
      P_{n_j}\to P \quad\text{in }L^{p'}_{\loc}(\R^3;\R^3),\qquad
      P_{n_j}^2\to P^2 \quad\text{in }L^2_{\loc}(\R^3;\R^3).  
    $$
  \end{prop}
  \begin{proof}
   As in the case of a bounded domain, one finds that
   $(P_n)$ is bounded and w.l.o.g. 
   weakly convergent to some $P\in Z$.  
   From $J'(P_n)\to 0$ and $P_n\wto P$ we infer for all bounded balls $B\subset\R^3$
   \begin{align*}
      o(1) 
      &= (J'(P_n)-J'(P))[(P_n-P)\ind_B] \\
      &= \int_B \big(\psi(x,P_n)-\psi(x,P)\big)\cdot (P_n-P) + \lambda^{-1}(P_n^2-P^2)\cdot (P_n-P)\,dx 
      - \int_B (P_n-P) \cdot \mathcal R(P_n^1-P^1)\,dx \\
      &= \int_B \big(\psi(x,P_n)-\psi(x,P)\big)\cdot (P_n-P) + \lambda^{-1}|P_n^2-P^2|^2\,dx  \\
      &- \int_B (P_n-P) \cdot \mathcal R(P_n^1-P^1)\,dx     
      + \lambda^{-1} \int_B (P_n^2-P^2)\cdot (P_n^1-P^1)\,dx
      \qquad\text{as }n\to\infty.
   \end{align*}  
    We emphasize that for a fixed $n\in\N$ the last term does not necessarily vanish given that the integral
    is over $B$ and not $\R^3$. The compactness of $\ind_B \mathcal R$ established in
    \cite[Lemma~4.1~(i)]{EveqWeth_Dual} implies
    $$
      \lim_{n\to\infty} \int_B (P_n-P)  \cdot \mathcal R(P_n^1-P^1)\,dx = 0.
    $$
    The Div-Curl-Lemma \cite[Theorem~4.5.16]{KCChang} and
    $\nabla\cdot (P_n^1-P^1) = 0$, $ \nabla\times (P_n^2-P^2) = 0$  
    give
    $$
       \lim_{n\to\infty} \int_B (P_n^2-P^2)\cdot (P_n^1-P^1)\,dx = 0.         
    $$
    So we conclude
    $$
      \lim_{n\to\infty} \int_B \big(\psi(x,P_n)-\psi(x,P)\big)\cdot (P_n-P) + \lambda^{-1}|P_n^2-P^2|^2\,dx  
      = 0. 
    $$
    As in the proof of Proposition~\ref{prop:PScondition} we obtain, up to the choice of a suitable
    subsequence, $P_n\to P$ in $L^{p'}(B)$ and
    $P_n^2\to P^2$ in $L^2(B)$ as well as $J'(P)=0$.   
  \end{proof}

  \begin{prop}\label{prop:Evequozargument}
     Assume $P,Q\in L^{p'}(\R^3;\R^3)$ and (A4). Then
     $$
      \|P-Q\|_{p'}
      \les  \left(\int_{\R^3}  \big(\psi(x,P)-\psi(x,Q)\big)\cdot (P-Q)\,dx\right)^{\frac{1}{2}}   \cdot
      (\|P\|_{p'}+\|Q\|_{p'})^{\frac{2-p'}{2}}.  
     $$ 
  \end{prop}
  \begin{proof}
    For almost all $x\in\R^3$ and $s_1,s_2\geq 0$ we can find $\xi$ between $s_1,s_2$ such that    
    $$
      \big(\psi_0(x,s_1)-\psi_0(x,s_2)\big)(s_1-s_2)
      = \partial_s \psi_0(x,\xi)\, (s_1-s_2)^2
      \sim |\xi|^{p'-2}\, (s_1-s_2)^2
      \gtrsim (|s_1|+|s_2|)^{p'-2}\, (s_1-s_2)^2.          
    $$
    Here we used (A4). This implies
    \begin{align*}
      \int_{\R^3}  \big(\psi(x,P)-\psi(x,Q)\big)\cdot (P-Q)\,dx
      &=  \int_{\R^3} \big(\psi_0(x,|P|)P +   \psi_0(x,|Q|)Q\big)\cdot (P-Q) \,dx     \\
      &\geq  \int_{\R^3} \Big(\psi_0(x,|P|)|P| +   \psi_0(x,|Q|)|Q| - \psi_0(x,|P|)|Q|-\psi_0(x,|Q|)|P|\Big)
      \,dx
      \\
      &= \int_{\R^3} \big(\psi_0(x,|P|)-\psi_0(x,|Q|)\big)(|P|-|Q|) \,dx \\
      &\gtrsim  \int_{\R^3}  (|P|+|Q|)^{p'-2}(|P|-|Q|)^2 \,dx \\
      &\gtrsim  \int_{\R^3}  (|P|+|Q|)^{p'-2}|P- Q|^2 \,dx.
    \end{align*}
    So H\"older's inequality gives
    \begin{align*}
      \|P-Q\|_{p'}
      &= \||P-Q|(|P|+|Q|)^{\frac{p'-2}{2}}\cdot (|P|+|Q|)^{\frac{2-p'}{2}} \|_{p'} \\
      &= \||P-Q|(|P|+|Q|)^{\frac{p'-2}{2}} \|_2 \| (|P|+|Q|)^{\frac{2-p'}{2}}
      \|_{\frac{1}{\frac{1}{2}-\frac{1}{p'}}}   \\
      &= \left(\int_{\R^3} |P-Q|^2(|P|+|Q|)^{p'-2}\,dx\right)^{\frac{1}{2}}   \cdot
      \|   |P|+|Q| \|_{p'}^{\frac{2-p'}{2}}   \\
      &\les  \left(\int_{\R^3}  \big(\psi(x,P)-\psi(x,Q)\big)\cdot (P-Q)\,dx\right)^{\frac{1}{2}}   \cdot
      (\|P\|_{p'}+\|Q\|_{p'})^{\frac{2-p'}{2}}.   
    \end{align*}
  \end{proof}
  
  \begin{thm}\label{thm:dualR3}
    Assume~\eqref{eq:assumptions_fullspace}. Then $J$ has a ground state and infinitely many geometrically
    distinct critical points. 
  \end{thm}
  \begin{proof}
    We apply Theorem~\ref{thm:CPTheorem} to the infinite-dimensional Banach space $Z=X^{p'}\oplus (Y^2\cap
    Y^{p'})$. The assumptions $(G_1),(G_2)$ of that theorem are straightforward to check for $Z^-:=\{0\}$. As to $(G_3)$,
    for any given $m\in\N$  choose $Z_m:=\spa\{\phi_1,\ldots,\phi_m\}\subset Z$ where 
    $$
      \hat \phi_j(\xi) := \chi_j(|\xi|^2)\vecIII{-\xi_1}{\xi_2}{0}
      \quad\text{where } \chi_j\in C_0^\infty(\R) \text{ and }\emptyset \subsetneq\supp(\chi_i)\subset
      (\omega^2\eps_0\mu_0+j,\omega^2\eps_0\mu_0+j+1).
    $$
    Then $\{\phi_1,\ldots,\phi_m\}$ is a linearly independent set of divergence-free Schwartz functions.
    Moreover, all elements of $Z_m$ are divergence-free, so $Z_m\subset X^{p'}$.
    Exploiting that all norms are equivalent on finite-dimensional spaces, one finds a 
    $c_m>0$ such that 
    $$
      \int_{\R^3} \mathcal R(P^1) \cdot P^1\,dx
      = \int_{\R^3} \mathcal R(P) \cdot P\,dx 
      = \int_{\R^3} \frac{|\hat P(\xi)|^2}{|\xi|^2-\omega^2\eps_0\mu_0}\,d\xi
      \geq c_m \|P\|^2 \qquad\text{for all }P\in Z_m. 
    $$
    Furthermore,
    $$
      \frac{1}{2\lambda}\int_{\R^3} |P^2|^2\,dx = 0  \qquad\text{for all }P\in Z_m. 
    $$    
    These facts imply $J(P)\to -\infty$ uniformly as $P\in Z_m,\|P\|\to\infty$ and $(G_3)$ is proved.
    Finally, we verify $(G_4)$ following \cite[Lemma~3.2]{EVEQ}. Assume we have two PS-sequences $(P_n),(Q_n)$
    for $J$. In the case 
    $$
      \int_{\R^3} (P_n^1-Q_n^1)\cdot \mathcal R(P_n^1-Q_n^1)\,dx \to 0
    $$
    we get
     \begin{align*}
      o(1) 
      &= (J'(P_n)-J'(Q_n))[P_n-Q_n] \\
      &= \int_{\R^3} \big(\psi(x,P_n)-\psi(x,Q_n)\big)\cdot (P_n-Q_n) + \lambda^{-1}|P_n^2-Q_n^2|^2 \,dx 
      - \int_{\R^3} (P_n^1-Q_n^1)\cdot \mathcal R(P_n^1-Q_n^1)\,dx \\
      &= \int_{\R^3} \big(\psi(x,P_n)-\psi(x,Q_n)\big)\cdot (P_n-Q_n) + \lambda^{-1}|P_n^2-Q_n^2|^2 \,dx
      + o(1)
      \qquad\text{as }n\to\infty,
   \end{align*} 
   which implies $\|P_n-Q_n\|\to 0$ in view of  Proposition~\ref{prop:Evequozargument}. In the complementary
   case the nonvanishing property of $\mathcal R$ from \cite[Theorem~2.3]{EveqWeth_Dual} gives, for a suitable ball $B\subset\R^3$, a positive number $\zeta$ and $x_n\in\Z^3$, that $\tilde P_n=P_n(\cdot-x_n), \tilde
   Q_n:=Q_n(\cdot-x_n)$ are still PS-sequences of $J$ with 
   $$
    \int_{B} |\tilde P_n^1-\tilde Q_n^1|^{p'} \,dx \geq \zeta>0  \qquad\text{for all }n\in\N. 
  $$
   Here, the $\Z^3$-periodicity of $f$ is used. By Proposition~\ref{prop:LocalCompactness} 
   we can select a subsequence, still denoted by $\tilde P_n,\tilde Q_n$ with weak limits  $P,Q\in\cK$
   satisfying 
   $$
    \int_{B} |P^1-Q^1|^{p'} \,dx \geq \zeta>0. 
  $$
  Hence, $P,Q\in\cK,P\neq Q$ and thus, by weak convergence,
  $$
    \limsup_{n\to\infty} \|P_n-Q_n\| \geq \|P-Q\| \geq \kappa.
  $$ 
  So $J$ is PS-attracting and $(G_4)$ is proved.
  So all assumptions of Theorem~\ref{thm:CPTheorem} hold, hence any $\cK$-decomposition
  of $J$ must be infinite. In particular, by Remark~\ref{rem:Abstract}(b), there are  
  infinitely many periodic orbits in $\cK$, which proves the claim. 
  \end{proof}
  
  \textbf{Proof of Theorem~\ref{thm:R3}:} It suffices to combine the existence of
  infinitely many periodic orbits in $\cK$ from Theorem~\ref{thm:dualR3} with the fact that critical points
  of $J$ yields critical points of $I$ by Lemma~\ref{lem:equivalenceR3}. \qed

\appendix

\section{Comments on the Dirichlet problem}     \label{sec:Dirichlet}

Only few adjustments are necessary to treat the Dirichlet problem in the same way as the Neumann
problem.  Essentially it suffices to  modify the function spaces. The linear
 Dirichlet boundary value problem reads
 \begin{equation} \label{eq:LinBVPD}
      \nabla\times \big(\mu(x)^{-1}\nabla\times E\big)
      - \lambda \eps(x) E = \eps(x) g \quad\text{in }\Omega,\qquad\quad
      E\times\nu = 0 \quad\text{on }\partial\Omega.
 \end{equation}
 In contrast to the Neumann problem,  the right function spaces for $E$ now is 
$$
   \mathcal V_0 :=  \Big\{ E_1\in \mathcal H_0: \int_\Omega
       \eps(x) E_1\cdot  \nabla \Phi\,dx = 0 \text{ for all } \Phi\in C_0^1(\Omega)\Big\}
$$
where $\mathcal H_0$ is defined as the closure of test functions with
respect to the inner product $\skp{\cdot}{\cdot}$ defined in \eqref{eq:innerproduct}. In particular,
$\mathcal V_0\subset \mathcal V$ is a closed subspace, but it is also a closed subspace
of $H_0^1(\Omega;\R^3)$. This follows just as in the Neumann case with the aid of
\cite[Theorem~1.1]{FilPro}.
A weak solution of \eqref{eq:LinBVPD} satisfies
 \begin{equation*} 
  \int_\Omega  \mu(x)^{-1} (\nabla\times E) \cdot (\nabla\times \Phi)\,dx
  -  \lambda  \int_\Omega  \eps(x)E\cdot\Phi\,dx
  = \int_\Omega \eps(x)g\cdot\Phi\,dx 
  \qquad\text{for all }\Phi\in\mathcal V_0. 
\end{equation*} 
 We introduce
 \begin{align*}
    X_0^{p'}
    &:= \Big\{ E\in L^{p'}(\Omega;\R^3): \int_\Omega \eps(x)E\cdot  \nabla\phi\,dx = 0 \text{ for all
    }\phi\in W_0^{1,p}(\Omega)\Big\},\\
    Y_0^{p'}  &:= \big\{ \nabla u: u\in W_0^{1,p'}(\Omega)\big\}.
 \end{align*}
 %The spaces $\tilde{X}_0^{p'},\tilde{Y}_0^{p'},\tilde{\mathcal V}_0, \tilde{\mathcal W}_0$ are defined
 %similarly to $\tilde{X}^{p'},\tilde{Y}^{p'},\tilde{\mathcal V}, \tilde{\mathcal W}$ in the
 %previous section.
  Notice that $X_0^{p'}\supsetneq X^{p'}$ and $Y_0^{p'}\subsetneq Y^{p'}$ given that
 elements of $X_0^{p'}$ do not carry any information about the boundary behaviour in contrast to elements of $X^{p'}$.
 For instance, in the case of a ball $\Omega=\{x\in\R^3: |x|<1\}$ we have
 $E(x):= \eps(x)^{-1}(x_2,x_1,0)\in X_0^{p'}\sm X^{p'}$. Indeed, the vector field $\eps(x)E$
 is divergence-free in $\Omega$ with $\eps(x)E(x)\cdot\nu(x) = \eps(x)E(x)\cdot x = 2x_1x_2 \not\equiv
 0$ on $\partial\Omega$. As before we define
  \begin{align*}
    (Y_0^{p'})^{\perp_\eps} &:= \Big\{f\in L^{p}(\Omega;\R^3): \int_\Omega \eps(x)f\cdot g\,dx = 0 \text{ for all
    }g\in Y_0^{p'}\Big\},  \\
    (X_0^{p'})^{\perp_\eps} &:= \Big\{g\in L^{p}(\Omega;\R^3): \int_\Omega \eps(x)f\cdot g\,dx = 0 \text{ for all
    }f\in X_0^{p'}\Big\}.
  \end{align*}
  The Helmholtz Decomposition for the Dirichlet problem relies on the following
  result~\cite[Theorem~1]{AuscherQafsaoui}.
     
 \begin{prop}[Auscher, Qafsaoui] \label{prop:AuschQafII}
   Assume (A1),(A2) and $1<p<\infty$. Then, for any given $f\in L^{p'}(\Omega;\R^3)$ the boundary value
   problem
   \begin{equation*}% \label{eq:BVPAuschQafII}
     \nabla\cdot \big(\eps(x)(f+\nabla u)\big) = 0 \quad\text{in }\Omega 
   \end{equation*}
   has a unique weak solution in $u\in W_0^{1,p'}(\Omega)$. It satisfies 
   $\|\nabla u\|_{p'} \les \|f\|_{p'}$. 
 \end{prop}

  \begin{prop}  \label{prop:HelmholtzDecompositionD}
    Assume (A1),(A2) and $1<p<\infty$. Then we have
    $$
      L^{p'}(\Omega;\R^3)=X_0^{p'}\oplus Y_0^{p'}
      \qquad\text{with}\quad  (Y_0^p)^{\perp_\eps} = X_0^{p'}, Y_0^p = (X_0^{p'})^{\perp_\eps}.
    $$
  \end{prop}
  \begin{proof}
    The proof of $L^{p'}(\Omega;\R^3)=X_0^{p'}\oplus Y_0^{p'}$ is the same as the one of
    Proposition~\ref{prop:HelmholtzDecompositionN} up to replacing Proposition~\ref{prop:AuschQafI} by
    Proposition~\ref{prop:AuschQafII}.
%     From the uniqueness result in Proposition~\ref{prop:uniquenessD} we get $X_0^{p'}\cap Y_0^{p'}=\{0\}$ and thus $X_0^{p'}\oplus
%     Y_0^{p'}\subset L^{p'}(\Omega;\R^3)$. Indeed, any $E\in X_0^{p'}\cap Y_0^{p'}$ solves $\Delta E=0$ with
%     $E\in W_0^{1,p'}(\Omega)$, which implies $E=0$.
%     The projector onto $X_0^{p'}$ is given by $\Pi_0^{p'}(g):=g+\nabla H$ where $H$ is the
%     solution from Proposition~\ref{prop:uniquenessD} for $h_D:=g,\gamma_D:= 0$.
%     Indeed, for all $\phi\in W_0^{1,p}(\Omega)$ we then have
%     $$
%       \int_\Omega \Pi_0^{p'}\cdot  \nabla \phi\,dx
%       = \int_\Omega  g\cdot \nabla \phi\,dx
%       + \int_\Omega  \nabla H \cdot \nabla\phi\,dx
%       = 0.
%     $$
%     Then  $g = \Pi_0^{p'}(g) - \nabla H$ for all $g\in L^{p'}(\Omega;\R^3)$ implies
%     $$
%       L^{p'}(\Omega;\R^3)=X_0^{p'}\oplus Y_0^{p'}
%       \quad\text{with projectors}\quad
%       \Pi_0^{p'}:L^{p'}(\Omega;\R^3)\to X_0^{p'},\;
%       \id-\Pi_0^{p'}:L^{p'}(\Omega;\R^3)\to Y_0^{p'}.
%     $$
%     As in Proposition~\ref{prop:HelmholtzDecompositionN} we have $(Y_0^p)^{\perp_\eps} =X_0^{p'}$ and
%     $Y_0^p\subset (X_0^{p'})^{\perp_\eps}$ by definition, so it remains to show
    We only show $(X_0^{p'})^{\perp_\eps} \subset Y_0^p$. For $f\in (X_0^{p'})^{\perp_\eps}$ we have $f\in
    (X^{p'})^{\perp_\eps}$  thanks to $X_0^{p'}\supset X^{p'}$. So 
    Proposition~\ref{prop:HelmholtzDecompositionN} gives $f= \nabla F$ for some $F\in
    W^{1,p}(\Omega)$. We want to show that we even have $F\in W_0^{1,p}(\Omega)$ after substracting a
    suitable constant. To see this let $\phi\in C^1(\partial\Omega)$ be arbitrary 
    with zero average and let $u\in W^{1,p'}(\Omega)$ denote the unique weak solution of
    $$
      -\Delta  u =0 \quad\text{in }\Omega,\qquad
       \nabla u \cdot\nu = \phi \quad\text{on } \partial\Omega.
    $$
    This is possible by \cite[Theorem~9.2]{FabMedMit_BoundaryLayers}.
    Then $u$ is harmonic with $\eps(x)^{-1}\nabla u\in X_0^{p'}$. Hence, 
    $$
      0 
      = \int_\Omega \eps(x)\nabla F\cdot \eps(x)^{-1} \nabla u\,dx
      = \int_\Omega  \nabla F\cdot \nabla u\,dx
      = \int_{\partial\Omega} F\phi\,d\sigma.  
    $$
    As a consequence,
    $$
      0
      = \int_{\partial\Omega} F\phi \,d\sigma
      \qquad\text{whenever}\quad \int_{\partial\Omega} \phi\,d\sigma = 0.
    $$
    So $F$ is constant on $\partial\Omega$ and after substracting the constant, we find
    $f= \nabla F$ for some $F\in W_0^{1,p}(\Omega)$,
    so $f\in Y_0^p$. Here we used that $F\in W_0^{1,p}(\Omega)$ holds if and only if $F\in W^{1,p}(\Omega)$
    has zero trace \cite[p.315]{Brezis}. This finishes the proof. 
  \end{proof}
   
   %In the case $\eps(x) = \eps_0\in (0,\infty)$ this result can be found in \cite[Theorem
   %11.2]{FabMedMit_BoundaryLayers} for $C^1$-domains with certain extensions to the case of Lipschitz
   %domains. 
   As in the Neumann setting $\mathcal V_0$ inherits
   the embeddings from $H_0^1(\Omega;\R^3)$, which allows to set up Fredholm theory
   for the linear Dirichlet problem~\eqref{eq:LinBVPD} using
\begin{equation*}%\label{eq:embeddingsD}
  \mathcal V_0 \hookrightarrow X_0^p 
  \quad\text{boundedly for }1\leq p\leq 6
  \text{ and compactly for } 1\leq p<6.
\end{equation*}
   In this way one finds that Theorem~\ref{thm:linear_theoryN} admits a counterpart for the Dirichlet problem
   with $\mathcal V,\mathcal W,X^{p'},Y^{q}$ replaced by $\mathcal V_0,\mathcal W_0,X_0^{p'},Y_0^{q}$,
   respectively. Replacing the function spaces in the discussion  of the nonlinear problems the leads to
   existence results for infinitely many solutions of the nonlinear Dirichlet problem under the same
   assumptions (A1),(A2),(A3).  We close this section by a remark on the interpretation of the boundary
   condition $E\times\nu=0$ on $\partial\Omega$.
   
   \medskip

  \begin{rem}\label{rem:BC}% ~
  %\begin{itemize}
%     \item[(a)] Geometrically distinct bound states solutions do not originate from
%     each other by changes of coordinates that leave the energy functional invariant. For instance, in the
%     model case of a ball $\Omega=\{x\in\R^3: |x|<1\}$ and $f(x,E)=|E|^{p'-2}E$, the bound
%     states are not generated from each other by rotations.
    %\item[(a)]
      In the context of \eqref{eq:NLCurlCurlD}  the
      metallic boundary condition $E\times\nu=0$ on $\partial\Omega$ holds in the sense  
     $$ 
       \int_\Omega \Big( (\nabla\times\Phi)\cdot E - \Phi\cdot (\nabla\times E) \Big)\,dx = 0
       \qquad\text{for all } \Phi\in C^1(\ov\Omega;\R^3).
	$$
	In fact, this identity holds for all $E\in C_0^\infty(\Omega;\R^3)$ and hence, by density with
	respect to $\|\cdot\|$, for all $E\in \mathcal H_0$. It encodes the boundary condition given that, under
	suitable regularity assumptions, the integral equals, by the Divergence Theorem, 
	$$
	  \int_\Omega \nabla\cdot(\Phi\times E)\,dx 
	  = \int_{\partial\Omega} (\Phi\times E)\cdot \nu\,d\sigma
	  = \int_{\partial\Omega} (E\times \nu)\cdot \Phi\,d\sigma
	$$ 
	for the outer unit normal field $\nu:\partial\Omega\to\R^3$. This is analogous to the classical Dirichlet
	problem for the Laplacian where the zero trace boundary condition comes with the space $H_0^1(\Omega)$. 
	This motivates the name ``Dirichlet problem'' for \eqref{eq:NLCurlCurlD}.
	In the context of~\eqref{eq:NLCurlCurlN} the boundary condition $(\mu(x)^{-1}\nabla\times E)\times\nu = 0$
	on $\partial\Omega$ is encoded in the Euler-Lagrange equation for the functional $I$ over $\mathcal H$. So it
	shows up as a free boundary condition, which is analogous to the Neumann problem for the Laplacian.
% 	This is motivated by the following calculation assuming that
% the data $E_1,\Omega,\mu$ is smooth enough and $\phi\in C^1(\ov\Omega;\R^3)$ is arbitrary:
% \begin{align*}
%     0 
%     &=  \int_\Omega   (\nabla\times E_1)^T \mu(x)^{-1}(\nabla\times \phi)\,dx- \int_\Omega f\cdot \phi \,dx \\
%     &= \int_\Omega \nabla\cdot\Big( (\mu(x)^{-1}\nabla\times E_1) \times \phi\big) + \mathcal L E_1\cdot \phi
%     \,dx - \int_\Omega f\cdot \phi \,dx \\ 
%     &= \int_{\partial\Omega} \big((\mu(x)^{-1}\nabla\times E_1) \times \phi\big) \cdot\nu\,d\sigma + 
%       \int_\Omega  \big( \mathcal L E_1  - f\big)\cdot \phi\,dx   \\
%     &= \int_{\partial\Omega} \phi\cdot \big(\nu\times (\mu(x)^{-1}\nabla\times E_1)\big)\,d\sigma
%     +   \int_\Omega  \big(  \mathcal L E_1  - f\big)\cdot \phi\,dx. 
%%   \end{align*}   
%  	 \item[(b)] The standard reference for the Nehari manifold is the paper by Szulkin and Weth~\cite{SzuWet}.
%  	 However, their main existence results about ground states from Theorem~12 and Theorem~13 in this paper do
%  	 not apply in our analysis of the dual problems for $\omega\neq 0$ given that.
%      \item[(c)] Our assumptions imply that the nonlinearity $f(x,\cdot)$ is odd, so the
%      functionals showing up in our dual approach are even. \r{So the existence of infinitely many solution may
%      for instance be deduced from the Symmetric Mountain Pass Theorem \cite[Theorem~6.5]{Struwe}. 
%      Szulkin: Ljusternik-Schnirelman}
%  \end{itemize}
  \end{rem}

\section{Some technical results}

\begin{lem}\label{lem:curlfree}
  Assume (A1),(A2) and $1\leq p<\infty$. If $f\in L^p(\Omega;\R^3)$ satisfies $\nabla\times f=0$ in
  the distributional sense, then $f=\nabla u$ for some $u\in W^{1,p}(\Omega)$.
\end{lem}
\begin{proof}
  By assumption we have  
  $$
    \int_\Omega f\cdot (\nabla\times\Phi)\,dx = 0\qquad\text{for all }\Phi\in C_0^\infty(\Omega;\R^3).
  $$
  Take mollifiers $(\eta_\tau)_{\tau>0}$ that form a smooth approximation of the
  identity with $\supp(\eta_\tau)\subset B_\tau(0)$ and define 
  $$
    f_\tau(x):=(\eta_\tau\ast f)(x):=\int_\Omega \eta_\tau(x-y)f(y)\,dy.
  $$  
  For any given $\Phi\in C_0^\infty(\Omega;\R^3)$ we have  $\Phi_\tau\in C_0^\infty(\Omega)$  for all positive
  $\tau<\dist(\supp(\Phi),\partial\Omega)$. This implies, using integration by parts, 
   $(\nabla\times\Phi)_\tau=\nabla\times\Phi_\tau$ as well as
  $$
      \int_\Omega (\nabla\times f_\tau)\cdot  \Phi \,dx
      = \int_\Omega f_\tau\cdot (\nabla\times\Phi)\,dx
      = \int_\Omega f\cdot (\nabla\times\Phi)_\tau\,dx
      =  \int_\Omega f\cdot (\nabla\times\Phi_\tau)\,dx
      \stackrel{\eqref{eq:orthogonalityN}}= 0.
    $$
    Given that the test function $\Phi$ is arbitrary as long as $\tau<\dist(\supp(\Phi),\partial\Omega)$,
    we conclude that for any given strictly contained ball $B\subset\subset \Omega$ we have $\nabla \times
    f_\tau =0$ in $B$ provided that $0<\tau<\dist(B,\partial\Omega)$. Since $f_\tau$ is smooth and any such
    ball is simply connected, we have 
    $$
      f_\tau|_B = \nabla u_{B,\tau} \quad\text{for a unique } u_{B,\tau}\in C^\infty(B)
      \text{ with }\int_{B} u_{B,\tau}\,dx = 0 \qquad\text{where }0<\tau<\dist(B,\partial\Omega).
    $$ 
    By Poincar\'{e}'s Inequality there is a constant $C=C(B)$ depending on $B$ such that 
    $$
      \|u_{B,\tau} - u_{B,\delta}\|_{W^{1,p}(B)} 
      \leq C\|\nabla u_{B,\tau} - \nabla u_{B,\delta}\|_{L^p(B)}
      = \|f_\tau-f_\delta\|_p
      \to 0\quad\text{as }\tau,\delta\to 0^+.
    $$
    We conclude $u_{B,\tau}\to u_B$ for some $u\in W^{1,p}(B)$ and it is standard to verify 
    $$
      f =\nabla u_B\quad\text{on }B.
    $$ 
    If $(\chi_i)_{i\in\N}$ is a partition of unity subordinate to some
    open cover $\Omega=\bigcup_{i\in\N} B_i$ with balls $B_i\subset \Omega$, it is straightforward to check
    $$
      f=\nabla u \quad\text{on } \Omega \qquad \text{where } 
      u:= \sum_{i\in\N} \chi_i F_{B_i}.
    $$
    Since $f\in L^p(\Omega)$, we have $u\in W^{1,p}(\Omega)$,  which is all we had to show.  
\end{proof}

 \begin{prop}\label{prop:Apsi}
    Assume that $f:\Omega\times\R^3\to\R^3$ satisfies~(A3). Then
    $\psi(x,\cdot):=f(x,\cdot)^{-1}$ exists for almost all $x\in\Omega$ and satisfies (A3').
  \end{prop}
  \begin{proof}
    By assumption (A3), for almost all $x\in\Omega$ the function $z\mapsto f_0(x,z)$ is
    positive, differentiable and increasing on $(0,\infty)$ with $f_0(x,z)\to 0$ as $z\to 0$ and
    $f_0(x,z)\to +\infty$ as $z\to\infty$. In particular, $f_0(x,\cdot):[0,\infty)\to [0,\infty)$ admits a positive,
    differentiable and increasing inverse $\psi_0(x,\cdot):=f_0(x,\cdot)^{-1}$ for such $x\in\Omega$.
    Moreover, $z\mapsto z^{-1} \psi_0(x,z)$ is decreasing on $(0,\infty)$ because $s\mapsto s^{-1}
    f_0(x,s)$ is increasing on $(0,\infty)$. This implies $f(x,\cdot)^{-1}=\psi(x,\cdot)$ where
    $\psi(x,P):=\psi_0(x,|P|)|P|^{-1}P$ and the claimed properties of $\psi_0$ except \eqref{eq:psi_estimate}.

	\medskip

	To prove \eqref{eq:psi_estimate} note that (A3) implies $z:= f_0(x,s)  \sim  s^{p-1}$ and thus $\psi_0(x,z) =
	s \sim  z^{p'-1}$. This implies the second inequality in \eqref{eq:psi_estimate}.
	%Define $F(x,E):=\int_0^{|E|} f_0(x,t)\,dt$ and $J(x,P):=\int_0^{|P|}J_0(x,t)\,dt$.
	Furthermore, by differentiation we find the identity
	$$
	  \int_0^{f_0(x,s)} \psi_0(x,t)\,dt + \int_0^s f_0(x,t)\,dt = sf_0(x,s)
	  \quad\text{for all }s\geq 0
	$$
	and conclude with the aid of assumption~(A3)
	$$
	  \int_0^z \psi_0(x,t)\,dt-\frac{1}{2} \psi_0(x,z)z=\frac{1}{2} f_0(x,s)s - \int_0^s f_0(x,t)\,dt
	  \gtrsim  s^p \sim z^{p'}.
	$$
	This provides the first inequality in~\eqref{eq:psi_estimate} and the claim is proved.
  \end{proof}

\bibliographystyle{abbrv}
\bibliography{biblio}

\end{document}